\documentclass[11pt,oneside]{amsart}

\usepackage{setspace}

\usepackage{amssymb}   
\usepackage{amsmath}
\usepackage{mathrsfs}
\usepackage{stmaryrd}
\usepackage{textcomp}
\usepackage{amsbsy}
\usepackage{cite}

\usepackage{geometry}
\usepackage{graphicx}
\usepackage{picture}
\usepackage{array}
\usepackage{epic}
\usepackage{enumerate}
\usepackage[all,cmtip]{xy}
\usepackage[percent]{overpic}

\usepackage{datetime}
\renewcommand{\dateseparator}{-}
\renewcommand{\today}{\the\year \dateseparator \twodigit\month
\dateseparator \twodigit\day}

\usepackage
[pdfauthor={Ethan J Street},
 pdftitle={The Title},
 bookmarks=false]
{hyperref}

\title[Recursive Relations for Moduli of Parabolic Bundles]{Recursive Relations in the Cohomology Ring of Moduli Spaces of Rank 2 Parabolic Bundles on the Riemann Sphere} 
\author{Ethan J. Street}

\newtheorem{thm}{Theorem}[section]

\theoremstyle{plain}
\newtheorem{prop}[thm]{Proposition}
\newtheorem{lemma}[thm]{Lemma}
\newtheorem{cor}[thm]{Corollary}

\theoremstyle{definition}
\newtheorem{definition}[thm]{Definition}

\newtheorem{remark}[thm]{Remark}
\newtheorem{example}[thm]{Example}

\newtheorem*{nonumnotation}{Notation}

\numberwithin{equation}{section}
\numberwithin{thm}{section}

\newcommand{\abs}[1]{\left\lvert{#1}\right\rvert}
\newcommand{\norm}[1]{\left\|{#1}\right\|}

\newcommand{\oo}{\infty}
\newcommand{\eps}{\epsilon}
\newcommand{\CP}{\mathbb{C}P}
\newcommand{\su}{\mathfrak s \mathfrak u}
\newcommand{\so}{\mathfrak s \mathfrak o}
\newcommand{\Id}{\mathrm{Id}}
\newcommand{\Ad}{\mathrm{Ad}}
\newcommand{\ad}{\mathrm{ad}}
\newcommand{\nin}{\notin}

\newcommand{\into}{\hookrightarrow}
\renewcommand{\obar}{\overline}
\newcommand{\ubar}{\underline}
\newcommand{\wt}{\widetilde}
\newcommand{\wh}{\widehat}
\newcommand{\id}{\mathrm{id}}
\newcommand{\pair}[2]{\left\langle{#1},{#2}\right\rangle}
\newcommand{\floor}[1]{\left\lfloor{#1}\right\rfloor}
\newcommand{\PAR}{\mathrm{par}}
\newcommand{\point}{\mathrm{pt}}

\newcommand\orb{\mathrm{orb}}
\newcommand\PD{\mathrm{PD}}

\DeclareMathOperator{\Aut}{Aut}

\DeclareMathOperator{\hol}{hol}
\DeclareMathOperator{\End}{End}

\DeclareMathOperator{\Mod}{Mod}

\DeclareMathOperator{\trace}{Tr}

\DeclareMathOperator{\SO}{SO}

\DeclareMathOperator{\SU}{SU}
\DeclareMathOperator{\U}{U}
\DeclareMathOperator{\PU}{PU}

\DeclareMathOperator{\diag}{diag}
\DeclareMathOperator{\sech}{sech}
\DeclareMathOperator{\Tr}{Tr}

\DeclareMathOperator{\rank}{rk}
\DeclareMathOperator{\Pic}{Pic}

\newcommand{\cB}{{\mathcal B}}

\newcommand{\CC}{{\mathbb C}}

\newcommand{\QQ}{{\mathbb Q}}
\newcommand{\RR}{{\mathbb R}}

\newcommand{\ZZ}{{\mathbb Z}}

\begin{document}

\maketitle

\centerline{\bf Abstract}

\vspace{0.3in}

We study the singular cohomology of the moduli space of rank 2 parabolic bundles on a Riemann surface where the weights are all 1/4. We give a formula, based on work of Boden, for the Poincar\'e polynomial of this moduli space in general, and deduce a complete set of relations in the cohomology ring for the case when the genus is zero. These relations are recursive in the number of parabolic points (which must be odd), and are analogous to the relations appearing in the cohomology ring for the non-parabolic story, which are recursive in the genus. Our methods use techniques from a paper of Weitsman to reduce to a linear algebra problem, which we then solve using the theory of orthogonal polynomials and a continued fraction expansion for the generating function of the Euler numbers.
\section{Introduction}

Given a compact Riemann surface $\Sigma$ of genus $g$, it is well known that the moduli space of stable, rank 2 holomorphic bundles $\mathcal{E}$ with given fixed odd degree determinant line bundle $L$ is a smooth compact manifold $\mathcal{M}^0_g(2,1)$ of real dimension $6g-6$ \cite{AB}. It is naturally a K\"ahler manifold, in fact a rational projective variety over $\CC$, and has been studied extensively. This space can also be studied from the point of view of representations of the fundamental group of $\Sigma$ into the group $\SU(2)$. Its cohomology ring has been studied in great detail and presentations for it have been given by several authors \cite{king_newstead, siebert_tian}. In this paper we are interested in the analgous problem for holomorphic vector bundles on an orbifold surface $\check\Sigma$. The relevant object of study is the moduli space of vector bundles over $\Sigma$ with parabolic structure on a collection of points. There is a natural notion of stability for such bundles and, under suitable genericity assumptions on the weights, a corresponding smooth, compact moduli space. We compute the cohomology ring of this moduli space for a specific symmetric choice of weights and in the case of rank 2 bundles over the Riemann sphere.

For a surface of arbitrary genus $g$ with $n$ parabolic points and arbitary weights, this moduli space can be quite a bit more complicated than $\mathcal{M}^0_g(2,1)$. Our approach will be to use the classical theorem of Mehta and Seshadri \cite{mehta_seshadri} that it is isomorphic to a suitable moduli space of representations of the orbifold fundamental group $\pi^\orb_1(\check\Sigma)$. Part of our motivation in the work here is to understand a question in instanton Floer homology, and the setup of that theory necessitates that we restrict attention here to the case that all the parabolic weights are $1/4$ and $n$ is odd. In this case, the problem becomes more tractable. We denote the moduli space of parabolic bundles (with fixed determinant) in this special case $\mathcal{M}^0_{g,n}$. The theorem of Mehta and Seshadri shows that $\mathcal{M}^0_{g,n}$ is isomorphic to a certain representation variety $\mathcal{R}_{g,n}$ of $\pi^\orb_1(\check\Sigma)$ into $\SU(2)$. We shall prove the following result for the genus 0 case:

\begin{thm} \label{thm_cohom_ring_genus_0}
Let $n=2m+1\geq 3$. The cohomology ring $H^*(\mathcal{R}_{0,n};\QQ)$ is zero in odd dimensions, and is generated as a $\QQ$-algebra by the class $\alpha$ of twice the natural symplectic form on $\mathcal{R}_{0,n}$ and degree two classes $\delta_1,\ldots,\delta_n$ corresponding to the parabolic points, satisfying $\delta_j^2=\delta_k^2$ for all $j,k$. Let $n=2m+1$, define $\beta:=\delta_k^2\in H^4(\mathcal{M}_{0,n})$ and for each $m\geq 0$ define the polynomial $r_{0,2m+1}(\alpha,\beta)$ for $n\geq 3$ recursively via the relation:
$$r_{0,2m+3}=\alpha\cdot r_{0,2m+1}-m^2\beta\cdot r_{0,2m-1}$$
with $r_{0,1}=1$, $r_{0,3}=\alpha$. Then for each subset $J\subset\{1,\ldots,n\}$ with $\abs{J}=s\leq m$, the polynomial
$$R_{0,2m+1}^J=r_{0,2m-2s+1}\cdot\prod_{k\in J}\delta_k$$
is a relation in $H^*(\mathcal{R}_{0,n};\QQ)$. This set of relations, along with $\delta_k^2=\delta_j^2$ is complete.
\end{thm}

We shall find it necessary to consider both points of view of the space $\mathcal{R}_{g,n}$, as a space of parabolic bundles, and a space of representations. As a go between, we will also give a description of this moduli space in terms of flat connections, where the natural symplectic form on it is easiest to describe.

To get a handle on the topology of $\mathcal{R}_{g,n}$, we begin by determining its Betti numbers. Methods for computing the Poincar\'e polynomials of moduli spaces of parabolic bundles have been described by several authors. Hans Boden, in his thesis \cite{boden_parabolic}, provides a comprehensive account of an adaptation of the equivariant Morse theory method from \cite{AB} to the parabolic case that achieves this. We leverage a non-closed form formula from that exposition to read off the Poincar\'e polynomial of our $\mathcal{R}_{g,n}$.

In order to determine generators and relations for the cohomology ring, we adapt the method in \cite{weitsman_cohom}, which can be summarized as follows. Lying over $\mathcal{R}_{g,n}$, there is a natural $\U(1)$ bundle coming from each of the marked points. The first Chern classes of these bundles, along with the classical generators described by Atiyah and Bott fully generate the cohomology ring. These bundles are trivial outside certain real codimension two submanifolds given by representations where two generators are mapped to equal or opposite elements of $\SU(2)$. This allows us to identify representatives for the Poincar\'e duals of these generators. Each of these submanifolds behaves like a 2-sphere bundle over a copy of the moduli space for two fewer parabolic points. We use this fact to analyze the intersection numbers of the corresponding homology classes recursively in the number $n$ of parabolic points. We can now reduce the problem to linear algebra in genus 0. Remarkably, we are able to solve this linear algebra problem using techniques from the theory of orthogonal polynomials and an identity involving the so called \emph{Euler numbers}.

The results of this paper are an analogue of the results of \cite{siebert_tian} for parabolic bundles. Part of the motivation for this investigation was to develop a similar analogue for the results of Mu\~noz in \cite{munoz_cohom} for singular instanton Floer homology. This work has been completed, using the results proved here, and will appear in a forthcoming paper \cite{street_floer}.

\section{Preliminaries}

In this section we define the moduli spaces whose cohomology we are interested in and record some of the basic facts known about them. Specifically, we give the three points of view - flat connections, representations, and parabolic stable bundles - we shall need, and describe a canonical set of generators for the cohomology ring. Boden's thesis \cite{boden_parabolic} gives an excellent introduction to the various points of view of the moduli spaces we are interested in. Our treatment of parabolic bundles comes from the discussion there, while our treatment of flat connections is modelled on the presentation in \cite{weitsman_cohom}.

\subsection{Moduli Spaces and the Symplectic Structure} \label{subsec_mod_space}

We first describe the three moduli spaces and explain briefly the isomorphisms between them. We then define the symplectic structure on the moduli space of flat connections.

\vspace{12pt}

\def\GC{\mathcal{G}^\text{c}}
\def\GCp{\mathcal{G}^\text{c}_\text{par}}
\def\GCr{\obar{\mathcal{G}}\vphantom{X}^\text{c}}
\def\GCpr{\obar{\mathcal{G}}\vphantom{X}^\text{c}_\text{par}}
\def\gCr{\obar{G}\vphantom{X}^\text{c}}
\def\Cs{\mathcal{C}^\mathrm{s}}

\noindent\textbf{Stable Parabolic Bundles.} Our starting point is a compact Riemann surface $\Sigma$ of genus $g$ with $n$ marked points $x_1,\ldots,x_n$.
\begin{definition} \label{def_par_bundle}
A \emph{parabolic} vector bundle $\mathcal{E}$ on $\Sigma$ with respect to the marked points is holomorphic vector bundle along with the data of a descending filtration at each $x_k$:
$$E_{x_k}= F^{(0)}_k\supset F^{(0)}_k\supset\dots\supset F^{(s_k)}_k$$
and weights $t^{(0)}_k<\ldots<t^{(s_k)}_k$ in $[0,1]$.
\end{definition}
There are notions of degree and stability for parabolic bundles. For a parabolic bundle $\mathcal{E}\to\Sigma$, set
$$d_k^{(j)}=\rank(F^{(j)}_k/F^{(j+1)}_k).$$
\begin{definition} \label{def_par_degree}
The \emph{parabolic degree} of $\mathcal{E}$ is the quantity
$$\deg^\PAR(\mathcal{E}):=\deg(\mathcal{E})+\sum_{k=1}^n\sum_{j=0}^{s_k}d_k^{(j)}t_k^{(j)}.$$
The \emph{slope} of $\mathcal{E}$ is the number
$$\mu^\PAR(\mathcal{E}):=\deg^\PAR(\mathcal{E})/\rank(\mathcal{E}).$$
\end{definition}
There are natural notions of sub- and quotient parabolic bundles (for more on definitions and facts regarding parabolic bundles and stability, we refer the reader to \S3 of \cite{boden_parabolic}).
\begin{definition} \label{def_par_stable}
The parabolic bundle $\mathcal{E}$ is called \emph{stable} if for any proper parabolic subbundle $\mathcal{F}$, one has $\mu^\PAR(\mathcal{F})<\mu^\PAR(\mathcal{E})$.
\end{definition}
For a suitable (that is, generic) fixed collection of weights, there is a smooth, compact moduli space of stable parabolic bundles. In this paper, we are concerned only with the case when $\mathcal{E}$ is rank two, and it will be convenient to assume that $\deg(\mathcal{E})=1$. There are then only two weights $t^{(0)}_k$ and $t^{(1)}_k$ at each marked point, and the filtration consists of a choice of line $F_k\subset\mathcal{E}_{x_k}$. We also make the additional assumption that the sum $t^{(0)}_k+t^{(1)}_k=1$, and set $t_k=t_k^{(0)}$. Denote the moduli space in this case by
$$\mathcal{M}_{g,n}(t_1,\ldots,t_n).$$
When the weights are understood, we will simply write $\mathcal{M}_{g,n}$. The space $\mathcal{M}_{g,n}$ admits a determinant map $\det$ to the space $\Pic^1(\Sigma)$ of isomorphism classes of degree one line bundles, which is a torsor for the Jacobian variety $J(\Sigma)$. Fixing a line bundle $\mathcal{L}\in\Pic^1(\Sigma)$, denote the fiber of $\det$ over $\mathcal{L}$ by $\mathcal{M}^0_{g,n}$ (the moduli space with fixed determinant). The Jacobian $J(\Sigma)$ also acts on $\mathcal{M}_{g,n}$ by tensor product and restricting this action gives a map
$$\mathcal{M}^0_{g,n}\times J(\Sigma)\to\mathcal{M}_{g,n}$$
which is a (connected, so nontrivial) degree $4^g$ covering.

\vspace{12pt}

\noindent\textbf{Representation Varieties.} Denote by $\Sigma^*$ the noncompact surface obtained by removing the $x_k$'s. Let $a_1,\ldots,a_{2g}$ be standard set of loops in $\Sigma$ generating the fundamental group so that their homology classes are a symplectic basis for $H_1(\Sigma)$ (that is, $a_j$ pairs with $a_{j+1}$) and for each removed point $x_k$ let $d_k$ denote a simple loop going once around the puncture. To get a relationship between representations of the fundamental group of $\Sigma^*$ and parabolic bundles with an \emph{odd} first Chern class, we look at the $\ZZ/2$ central extension $\wh\Gamma$ of the fundamental group $\Gamma=\pi_1(\Sigma^*)$ which has an extra order two central generator $\zeta$ and is determined by the single relation
$$\prod_{j=1}^g[a_{2j-1},a_2j]\cdot\prod_{k=1}^nd_k=\zeta$$
We consider a space of representations of $\wh\Gamma$ into $\SU(2)$. We bring the parabolic weights $t_k$ into the picture by only allowing representations to send the generator $d_k$ to an element of trace $2\cos(2\pi t_k)$. Let $\ubar t$ denote the weight vector $(t_1,\ldots,t_k)$

\begin{definition} \label{def_rep_variety}
The \emph{representation variety} $\mathcal{R}^\mathrm{odd}_{g,n}(\ubar t)$ for the weight vector $\ubar t$ is the quotient of the space of representations of
$$\rho:\wh\Gamma\to\SU(2)$$
such that $\Tr(\rho(d_k))=2\cos(2\pi t_k)$ and $\rho(\zeta)=-1$ under the action of conjugation by $\SU(2)$ on the target.
\end{definition}

Denote by $\wt{\mathcal{R}}^\text{odd}_{g,n}(\ubar t)$ space of representations before quotienting by conjugation. The action of $\SU(2)$ descends to one of $\PU(2)=\SU(2)/\{\pm 1\}$. This action is free and the quotient  $\mathcal{R}^\text{odd}_{g,n}(\ubar t)$ is a smooth, compact manifold only as long as the $t_k$'s are generic: for any collection of signs $\eps_k=\pm 1$, no sum $\Sigma_{k=1}^n\eps_k t_k$ may be an integer. This condition is exactly the one that precludes \emph{reducible} representations. Let $\ubar t$ denote the weight vector $(t_1,\ldots,t_k)$. We have the following important correspondence:

\begin{thm} \label{equiv_mod_spaces}
(Mehta, Seshadri, \cite{mehta_seshadri}) For a generic collection of weights $\ubar t$, there is a diffeomorphism 
$$\mathcal{M}^0_{g,n}(\ubar t)\cong\mathcal{R}^\text{odd}_{g,n}(\ubar t).$$
\end{thm}

The topology of these moduli spaces depends on the weights $t_k$. It is known (see \cite{boden_variations}, for example) that the genericity condition splits the space of weight vectors $\ubar t$ into chambers separated by codimension one walls, and that as the weight vector passes through a wall, the moduli space undergoes a well-understood monoidal transformation consisting of a blow up and blow down along a submanifold. The geometric application (singular rank 2 instanton Floer homology) we have in mind requires us to fix $t_k=1/4$ for all $k$. In this case we denote the representation space
$$\mathcal{R}_{g,n}:=\mathcal{R}^\mathrm{odd}_{g,n}(1/4,\ldots,1/4).$$
The space $\mathcal{R}_{g,n}$ is a quotient of the space $\wt{\mathcal{R}}_{g,n}$ consisting of $2g+n$-tuples $$(S_1,\ldots,S_{2g},T_1,\ldots,T_n)$$ of elements of $\SU(2)$ satisfying
$$\prod_{j=1}^g[S_{2j-1},S_{2j}]\cdot\prod_{k=1}^n T_k=-1.$$
and such that $\Tr(T_k)=0$, under the action of $\SU(2)$ by conjugation. The quotient is a smooth, compact manifold of real dimension $6g-6+2n$, as long as $n$ is an odd number $2m+1$, which we henceforth assume. For example, $\mathcal{R}_{0,3}$ consists of a single point: any representation may be conjugated to one sending the three generators to $\mathbf{i},\mathbf{j},\mathbf{k}\in\SU(2)$.

\begin{nonumnotation}
It will be convenient later onto use the notation $\ubar S$ and $\ubar T$ for the $2g$ and $n$ tuples of elements $(S_1,\ldots,S_{2g})$ and $(T_1,\ldots,T_n)$ of $\SU(2)$. Denote the equivalence class in $\mathcal{R}_{g,n}$ of a point $(S_1,\ldots,S_{2g},T_1,\ldots,T_n)$ in $\wt{\mathcal{R}}_{g,n}$ by $[S_1,\ldots,S_{2g},T_1,\ldots,T_n]$ or simply $[\ubar S,\ubar T]$.
\end{nonumnotation}

\vspace{12pt}

\noindent\textbf{Flat Connections.} The fact that isomorphism classes of flat connections and representations of the fundamental group are in bijective correspondence is an old idea in topology. In the presence of the parabolic points, the discussion of this correspondence becomes somewhat awkward and necessitates a careful discussion of exactly what kinds of connections we allow, which we now undertake. We essentially copy the setup for this from \cite{weitsman_toric_I}. In what follows, all maps and forms are $C^\oo$.

Recall that we have fixed a smooth vector bundle $E$ on $\Sigma$ with $c_1(E)=1$. This removes the possibility for reducible connections in the case that there are no marked points. We fix a hermitian inner product $\langle\cdot,\cdot\rangle$ on $E$, and it will be useful to introduce the corresponding principal $\U(2)$ bundle $P$. Denote by $\Ad P$ is the associated principal $\PU(2)$ bundle and $\Ad E$ the associated rank three real vector bundle via the isomorphism $\PU(2)\to\SO(3)$. The bundle $\Ad E$ sits naturally inside $\End E$ as the skew hermitian trace-free subspace. On the noncompact surface $\Sigma^*$, for each $k$ we identify a neighborhood $U_k$ of the $k$th puncture with the cylinder $(0,1)\times S^1$. This gives coordinates $(s_k,\theta_k)$ on $U_k$, with $\theta_k\in[0,2\pi]$. We want to study $\SU(2)$ connections on $E$, but $c_1(E)\neq 0$ so $E$ cannot be made to have structure group $\SU(2)$. We therefore use the usual technique of studying connections on $\Ad E$, but use a subgroup of the full gauge group of $\Ad E$ consisting of those automorphisms which are induced by determinant 1 gauge transformations on $E$. For each $k$, we fix a trivialization of $E$ over $U_k$ such that for $1\leq k\leq n$ the preferred line $F_k$ is spanned by the vector $(1,0)\in\CC^2$. With respect to these trivializations, connections on $\Ad E$ over can be identified with 1-forms with values in $\so(3)$. We denote by $\mathcal{A}$ the following space of connections on $\Sigma^*$ which are ``standard'' near the punctures:
\begin{equation} \label{conn_space}
\begin{split}
\mathcal{A}=&\Big\{A\text{ a }C^\oo,\;\SO(3)\text{ connection on }\Ad E\text{ such that: } \\
&\left.A(s_k,\theta_k)=\tfrac{1}{4}\ad\left(\begin{smallmatrix}i&0\\0&i\end{smallmatrix}\right)d\theta_k=\tfrac{1}{2}\left(\begin{smallmatrix}0&0&0\\0&0&-1\\0&1&0\end{smallmatrix}\right)d\theta_k\text{ for }s_k\geq 1/2\text{, }1\leq k\leq n\right\}
\end{split}
\end{equation}
These connections are fixed on the ends of $\Sigma$ in order to ensure that the corresponding connection has fixed holonomy around the punctures:
\begin{equation} \label{adj_hol}
\hol_{\partial U_k}(A)=\exp\left(2\pi\cdot\frac{1}{2}\left(\begin{smallmatrix}0&0&0\\0&0&-1\\0&1&0\end{smallmatrix}\right)\right)=
\left(\begin{smallmatrix}1&0&0\\0&-1&0\\0&0&-1\end{smallmatrix}\right)\in\SO(3).
\end{equation}
At the puncture $x_k$, the line $F_k\subset E_{x_k}$ picks out a real 2-dimensional subspace $H_k$ of $(\Ad E)_{x_k}$ corresponding to endomorphisms $h$ for which $h(F_k)\subset F_k^\perp$, which is actually a complex line by precomposing with complex scalars. Our space of connections is rigged so that the holonomy around $x_k$ is exactly the element which is the identity on the real line $\ell_k:=H_k^\perp$ and $-1$ on the plane $H_k$.

Given a unitary automorphism $g$ of $E$, there is an induced automorphism $\Ad g$ on $\Ad E$ which arises by conjugation by $g$. We define the gauge group $\mathcal{G}$ to be the space of smooth determinant 1 sections of $\Aut(E)$ of unitary automorphisms of $E$, and restrict to those automorphisms whose action preserves the end behavior of our 1-forms. Namely, we define
\begin{equation} \label{gauge_G}
\mathcal{G}=\left\{\begin{matrix}g\in\Aut(E)\text{: }\det(g)=1\text{, and }\forall\; k\text{, there is }z\in\CC\text{, }\abs{w}=1 \\
\text{ such that }g(s_k,\theta_k)=\left(\begin{smallmatrix}w&0\\0&-w\end{smallmatrix}\right)\text{, for }s_k\geq 1/2 \end{matrix} \right\}
\end{equation}
where we use the trivialization near each each puncture to identify $g$ with an $\SU(2)$-valued function. An element $g$ acts by pulling back connections via the automorphism $\Ad g$ on $\Ad E$. In other words a section $t$ of $\Ad E$ is $(g\cdot A)$-parallel if and only if $\Ad(g)\cdot t=g\circ t\circ g^{-1}$ is $A$ parallel. The gauge group is rigged so that $\Ad g$ preserves the orthogonal decomposition $H_k\oplus\ell_k$. We will think of $\mathcal{G}$ as acting on the bundle $\Ad E$ over the entire surface $\Sigma$, while $\mathcal{A}$ consists of connections only over $\Sigma^*$.

Denote by
$$\mathcal{A}_\text{flat}\subset\mathcal{A}$$
denote the subset of flat connections. It is not difficult to see that any smooth flat $\SO(3)$ connection on $\Sigma^*$ with the proper holonomy around the punctures will be gauge equivalent to one in $\mathcal{A}_\text{flat}$. Denote by $\mathcal{B}_\text{flat}$ the quotient $\mathcal{A}_\text{flat}/\mathcal{G}$. It is a standard fact that there is a diffeomorphism
$$\mathcal{B}_\text{flat}\cong\mathcal{R}_{g,n}.$$
It will be useful to understand this later when we begin analyzing the symplectic form in more detail. Fix a basepoint $z\in\Sigma$ away from the punctures and trivialize $E$ outside $z_0$. This gives a reduction of the structure group of $P|_{\Sigma\setminus\{z_0\}}$ to $\SU(2)$. With respect to the trivialization, a connection on $\Ad E$ over $\Sigma^*\setminus\{z_0\}$ becomes an $\so(3)$ valued 1-form $a$. An $\SU(2)$ connection $B$ on $E$ becomes an $\su(2)$ valued 1-form $b$, and the induced connection $\Ad B$ on $\Ad E$ has 1-form given by $\ad b$. The Lie algebra homomorphism $\ad:\su(2)\to\so(3)$ is an isomorphism so there is a unique $b$ for which $\ad b=a$, and $B$ is flat if and only if $A$ is. Letting $z$ be a new basepoint near $z_0$, the holonomy $\hol_z(A)$ gives a representation of $\pi_1(\Sigma^*\setminus\{z_0\})$ into $\SO(3)$ and $\hol_z(B)$ gives one to $\SU(2)$ which lifts the homomorphism $\hol_z(A)$. The holonomy of $B$ around a small loop $\gamma_0$ based at $z$ around $z_0$ must be $-1$ since the bundle doesn't extend across $z_0$ but this holonomy must lift the holonomy of $A$, which is the identity in $\SO(3)$. Hence, from a flat connection $A$ we get a representation of $\wh\Gamma\cong\pi_1(\Sigma^*\setminus\{z_0\})/\langle \gamma_{z_0}^2\rangle$ into $\SU(2)$ given by the holonomy of the lift $B$. What remains is to check that the holonomy of $B$ around the punctures has trace 0. To see this, we simply note that we are free to fix our trivialization of $E$ away from $z_0$ in a way agreeing with the fixed trivializations already chosen near each punctures. With respect to this, the 1-form of $A$ near the punctures has been fixed so that the holonomy around a small loop is the $3\times 3$ matrix in (\ref{adj_hol}). The two $\SU(2)$ lifts of this matrix are $\pm \mathbf{i}$, which have trace 0. Now, the action of the gauge group serves to conjugate the representation $\hol_z(B)$, hence we get the desired map $\mathcal{B}_\text{flat}\to\mathcal{R}_{g,n}$. For the rest of the paper, we fix the point $z_0$ and trivialization of $E$ on $\Sigma\setminus\{z_0\}$, writing $\su(A)$ for the unique $\SU(2)$ connection with $\Ad(\su(A))=A$ away from $z_0$ (which we were calling $B$).

\vspace{12pt}

\noindent\textbf{The Symplectic Structure.} The tangent space to a point $A\in\mathcal{A}$ is naturally identified with the space of 1-forms $a$ with values in the bundle $\so(\Ad E)\subset\End(\Ad E)$ of skew-adjoint endomorphisms, which are zero near the punctures. The subspace corresponding to $T_A\mathcal{A}_\text{flat}$ is the subspace satisfying $d_Aa=0$, the linearization of the flatness condition. Let
$$\trace(\cdot,\cdot):\so(3)\otimes\so(3)\to\RR$$
denote the invariant positive definite form $\trace(X,Y)=-\trace(X \circ Y)$, which gives an inner product on $\so(\Ad E)\subset\End(\Ad E)$. We then define a 2-form $\wt\omega$ on $\mathcal{A}_\text{flat}$ via:
\begin{equation} \label{def_symp_form}
\wt\omega(a\wedge b)=\frac{1}{4\pi^2}\int_\Sigma \trace(a\wedge b)
\end{equation}
Here, we use the natural composite of $\trace$ with the wedge product $\wedge$ on 1-forms. The tangent space to an equivalence class $[A]\in\mathcal{B}_\text{flat}$ is the quotient $T_A\mathcal{A}_\text{flat}/T_A(\mathcal{G}\cdot A)$. The key point is that $\omega$ is annihilated by $T_A(\mathcal{G}\cdot A)$. A vector in $T_A(\mathcal{G}\cdot A)$ is given by $d_A v\in\Omega^1(\Sigma^*)\otimes\so(\Ad E)$ for some $v\in\Omega^0(\Sigma^*)\otimes\so(\Ad E)$. We have:
$$\wt\omega(d_A v\wedge b)=\frac{1}{4\pi^2}\int_\Sigma \trace(d_A v,b)=-\frac{1}{4\pi^2}\int_\Sigma \trace(v,d_A b)=0$$
The second equality is due to Stoke's theorem ($\Sigma^*$ is not closed, but all functions involved are constant or zero near the boundary), and the third is because $b$ is infinitesimally flat ($d_A b=0$). Hence, $\wt\omega$ descends to a 2-form $\omega$ on the quotient $\cB_\text{flat}$. Nondegeneracy and closedness are standard results going back to \cite{AB}, where it is shown that $\omega$ is the symplectic form on $\mathcal{R}_{g,n}$ arising from an infinite dimensional symplectic reduction of $(\mathcal{A},\wt\omega)$ with moment map the curvature $F_\bullet:\mathcal{A}\to\Omega^2(\Sigma,\so(\Ad E))$.

\subsection{Universal Bundles and Generators for Cohomology} \label{subsection_univ_bundles}

Fundamental to the computation of the cohomology of $\mathcal{M}_g^0(2,1)$ is the knowledge that a certain canonical collection of cohomology classes generate the full rational cohomology ring. These classes were first described in \cite{AB}, where the methods of infinite dimensional equivariant Morse theory are used to show that they are generators. This technique was generalized in \cite{biswas_cohom} (using \cite{nitsure_cohom}) to the case of parabolic bundles, the essential aspects of which we review here.

The cohomology classes we describe all arise naturally from one or more ``universal bundles'' over a product of $\Sigma$ with a moduli space or classfiying space. This terminology is used for several distinct concepts in this circle of ideas, depending on the point of view taken of the moduli space of interest. We will take the flat connection point of view for this. Recall that $\mathcal{A}_\text{flat}$ is our space of flat connections on $\Ad E$ over $\Sigma^*$. Letting $\pi$ now denote the projection to $\mathcal{A}_\text{flat}\times\Sigma\to\Sigma$, there is an obvious tautological family of connections on $\pi^*(\Ad E)$, which on the slice through $A$ is just the connection $d_A$. This family is smooth in $\mathcal{A}_\text{flat}$ directions, and is preserved by the natural action of the smooth $\SU(2)$ gauge group $\mathcal{G}$. The center $\{\pm 1\}$ of $\mathcal{G}$ of constant scalar automorphisms acts trivially on $\mathcal{A}_\text{flat}$. In the holomorphic case, the center of the gauge group acted trivially on the base but not on the bundle; here we avoid this problem by passing to the adjoint bundle. The gauge group $\mathcal{G}$ acts on $\pi^*(\Ad E)$ via:
$$g\cdot(t,A,x)=(g_x\circ t\circ g_x^{-1},g\cdot A,x)$$
where $t\in(\Ad E)_x\subset\End(E_x)$. The center acts trivially, so this action descends to a free one of the quotient $\obar{\mathcal{G}}$. The quotient of $\pi^*\Ad E$ by $\obar{\mathcal{G}}$ is an $\SO(3)$ vector bundle $\mathbf{E}^\ad\to\mathcal{R}_{g,n}\times\Sigma$. Since an automorphism $g$ preserves the complex line $H_k\subset\Ad E_{x_k}$ for each $k$, it preserves the complex line subbundle $\pi^*(H_k)\subset \Ad E|_{\mathcal{A}_\text{flat}\times\Sigma}$. Hence, inside of $\mathbf{E}^\ad$, there is a subbundle $\mathbf{V}_k$ over $\mathcal{R}_{g,n}\times\{x_k\}$ coming from the quotient of $\pi^*(H_k)$ for each $k$. The pair $(\mathbf{E}^\ad,\{\mathbf{V}_k\})$ has a flat structure $\mathbf{A}$ in $\Sigma$ directions which moves continuously in $\mathcal{R}_{g,n}$ directions, and on $\{[A]\}\times\Sigma$ the bundle restricts to one with a flat connection on $\Sigma^*$ isomorphic to $A$, with the isomorphism carrying $(\mathbf{V}_k)_{[A]}$ to the subspace $H_k$.

\begin{definition} \label{def_univ_pair}
We shall call the pair $(\mathbf{E}^\ad,\{\mathbf{V}_k\})$ a \emph{universal pair}.
\end{definition}

\begin{thm} \label{thm_generating_set}
The rational cohomology ring $H^*(\mathcal{M}^0;\QQ)$ is generated as a $\QQ$-algebra by elements $p_1(\mathbf{E}^\ad)/h$, for $h\in H_*(\Sigma;\QQ)$, and the classes $c_1(\mathbf{V}_k)$.
\end{thm}
\begin{proof}
This a is straightforward consequence of the material in \cite{biswas_cohom}.
\end{proof}

\section{Betti Numbers of The Moduli Space} \label{sec_betti_numbers}

While \S\ref{subsection_univ_bundles} provides a generating set for the cohomology ring $H^*(\mathcal{R}_{g,n};\QQ)$, and \S\ref{section_relations} will describe relations among them, proving that these relations are a complete set for $g=0$ will require knowledge of the Betti numbers of $\mathcal{R}_{0,n}$. A convenient formula for these that is general enough to accommodate an arbitrary choice of parabolic weights appears in \cite{boden_parabolic}, and we can actually obtain formulas for the general case $g\neq 0$ without much additional work. To apply this to the specific case of equal weights treated here, we will need to interpret that formula a bit. To this end, we briefly review the methodology in \cite{boden_parabolic}.

As seen in \cite{nitsure_cohom}, the space of holomorphic structures $\mathcal{C}$ on $E\to\Sigma$ comes with a stratification analogous to the stratification described in \cite{AB} in the non-parabolic case. The subset of $\Cs\subset\mathcal{C}$ for which the fixed weights and 1 dimensional subspaces $F_k$ give a stable parabolic structure on the rank two bundle $E$ is an open subset. The point is that not all bundles in the complement are born equal: some unstable bundles are \emph{more} unstable than others. Given an unstable rank two bundle $\mathcal{E}$, there is a unique destabilizing line subbundle $\mathcal{F}$ sitting in a slope-decreasing short exact sequence
$$0\to\mathcal{F}\to\mathcal{E}\to\mathcal{Q}\to 0$$
Still following \cite{boden_parabolic}, we define $\lambda=\deg(\mathcal{F})$. Conceptually, the larger $\lambda$ is, the more unstable $\mathcal{E}$ is. We can get more data from $\mathcal{F}$ by recording whether it coincides with $F_k$ at each parabolic point. Namely, let $e_k=\dim(\mathcal{F}_{x_k}\cap F_k)$, and denote by $\ubar e$ the vector of the $e_k$'s, which are either 0 or 1. The invariant $\ubar e$ is a property of $\mathcal{E}$ and is exactly the \emph{intersection matrix} considered in \cite{nitsure_cohom}. The pair $(\lambda,\ubar e)$ is called the \emph{type} of $\mathcal{E}$, and we let $\mathcal{C}_{\lambda,\ubar e}$ denote the locally closed submanifold of $\mathcal{C}$ consisting of holomorphic structures on $E$ of type $(\lambda,\ubar e)$. The crucial point is that this strata for a fixed type is \emph{connected} (\cite{nitsure_cohom}, Prop. 3.5). The codimension $d_{\lambda,\ubar e}$ of the strata $\mathcal{C}_{\lambda,\ubar e}$ is
\begin{equation} \label{codim_form}
d_{\lambda,\ubar e}=2\lambda+n+g-1+\sum_k e_k
\end{equation}
which essentially appears in \cite{boden_parabolic} as equation (17).

\begin{remark}
In \cite{boden_parabolic}, the assumption is made that $\deg^\PAR(E)=0$, which is achieved by taking $\deg(E)=-n$. The moduli space in this case is isomorphic to ours as can be easily seen by comparing the representation varieties. For the rest of this section, we temporarily assume $\deg(E)=-n$ so as to use the formulas appearing there without attempting modification.
\end{remark}

The essential content of \cite{boden_parabolic} is that the stratification of $\mathcal{C}$ via the strata $\mathcal{C}_{\lambda,\ubar e}$ is \emph{equivariantly perfect} for the action of the complex gauge group $\GCp$. As a result, one can record the equivariant Poincar\'e polynomial of the total space as a sum over terms coming from the individual strata. These contributions can be computed exactly for the unstable strata, thus yielding a formula for the Poincar\'e polynomial for the quotient. We may now record the result in \cite{boden_parabolic} on the betti numbers of $\mathcal{M}^0_{g,n}$. For a topological space $X$, let $P_t(X)$ denote its Poincar\'e polynomial $\sum_d \text{rk}H_*(X)\cdot t^d$.

\begin{thm} \label{thm_poincare_poly}
The Poincar\'e polynomial of the representation variety $\mathcal{R}_{g,n}$ for $n\geq 1$ is given by:
\begin{equation} \label{poincare_poly}
P_t(\mathcal{R}_{g,n})=\frac{(1+t)^{2g-2}}{(1-t)^2}\left((1-t+t^2)^{2g}(1+t^2)^{n-1}-(1-t^2)\sum_{\lambda,\ubar e}t^{2d_{\lambda,\ubar e}}\right)
\end{equation}
where the sum ranges over all types $(\lambda,\ubar e)$ which are destabilizing.
\end{thm}
\begin{proof}
We know that $\mathcal{R}_{g,n}\cong\mathcal{M}^0_{g,n}$. Formula (\ref{poincare_poly}) is just equation (21) of \cite{boden_parabolic} divided by the Poincar\'e polynomial of the Jacobian $J(\Sigma)$, since the equation there is for the moduli space $\mathcal{M}_{g,n}$, which is a cohomologically trivial bundle over $\mathcal{M}^0_{g,n}$ with fiber homeomorphic to $J(\Sigma)$.
\end{proof}

\begin{example}
Let us illustrate how to use (\ref{poincare_poly}) to compute the Poincar\'e polynomial in a simple example: $\mathcal{R}_{0,3}$, which consists of the single equivalence class $[\mathbf{i},\mathbf{k},\mathbf{j}]$.
We need to understand the domain of the sum in the formula. The parabolic degree of the destabilizing line bundle $\mathcal{F}$ of a parabolic bundle $\mathcal{E}$ of type $(\lambda,\ubar e)$ is:
$$\deg^\PAR(\mathcal{F})=\lambda+\sum^3_{k=1}\left(\tfrac{1}{4}(1-e_k)+\tfrac{3}{4}e_k\right)=\lambda+\tfrac{3}{4}+\tfrac{1}{2}\sum_{k=1}^3e_k.$$
For such $\mathcal{F}$ to be destabilizing we need $\mu(\mathcal{F})\geq\mu(\mathcal{E})=3/2$. Letting $e=\sum_{k=1}^3e_k$, we therefore need
\begin{equation} \label{3pts_unstable}
4\lambda+2e\geq -3
\end{equation}
It is convenient to visualize the pairs $(\lambda,e)$ satisfying \ref{3pts_unstable}; see Figure \ref{fig_n3_unstable_grid}.

\begin{figure}[h] \label{fig_n3_unstable_grid}
\begin{overpic}[scale=1.5]{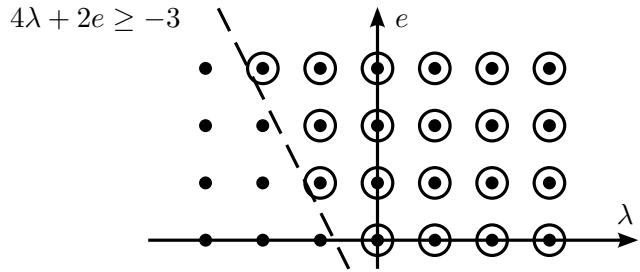}
\put(50,50){$e$}
\put(95,10){$\lambda$}
\put(-28,50){$4\lambda+2e\geq-3$}
\end{overpic}
\caption{The locus of unstable pairs $(\lambda,e)$ for $n=3$.}
\end{figure}

We see that we can take the sum to over all $e$ and $\lambda\geq-1$, except that we must add a contribution for $(-2,3)$ and subtract one for $(-1,0)$. We have:
\begin{align*}
\sum_{\lambda,\ubar e}t^{2d_{\lambda,\ubar e}}=&\;t^2-1+\sum_{\lambda\geq -1,\ubar e}t^{2(2\lambda+2+e)} \\
=&\;t^2-1+t^4\sum_{\lambda\geq -1}t^{4\lambda}\left(\sum_{e_1=0,1}t^{2e_1}\sum_{e_2=0,1}t^{2e_2}\sum_{e_3=0,1}t^{2e_3}\right)\\
=&\;t^2-1+(1-t^4)^{-1}(1+t^2)^3
\end{align*}
Since $g=0$ here, formula (\ref{poincare_poly}) reduces to
$$\frac{1}{(1-t^2)^2}\left((1+t)^2-(1-t^2)\sum_{\lambda,\ubar e}t^{2d_{\lambda,\ubar e}}\right).$$
Plugging in for the summation and applying simple algebra, we see that $P_t(\mathcal{R}_{0,3})=1$, as expected.
\end{example}

Focusing on the genus 0 case temporarily, we can use (\ref{poincare_poly}) to get a beautiful result which yields explicit formulae for the Poincar\'e polynomials.

\begin{prop} \label{prop_genus_0_poly_monomials}
Suppose $n=2,+1>3$. Then up to its middle dimension $n-3$, the Poincar\'e polynomial $P_t(\mathcal{R}_{0,n})$ equals the Poincar\'e polynomial of the graded algebra $\CC[\alpha,\beta,\delta_1,\ldots,\delta_n]/(\delta_i^2)$, where $\alpha$, $\beta$, $\delta_i$ have degrees 2, 4, and 2, respectively.
\end{prop}
\begin{proof}
The proof is an exercise in bookkeeping for the sum over strata in (\ref{poincare_poly}). For a type $(\lambda,\ubar e)$ to be destabilizing, we now require (againg letting $e$ denote the sum of the $e_k$'s):
\begin{equation} \label{unstable_ineq}
4\lambda+2e\geq -n
\end{equation}
Let $Q_n\subset\ZZ^2$ be the subset of the lattice points in the $(\lambda,e)$ plane satisfying (\ref{unstable_ineq}) and bounded by $0\leq e\leq n$. It can be approximated by the subset $\lambda\geq -m$, $0\leq e\leq n$; this approximation leaves out a small triangle of lattice points to the left of $\lambda=-m$, but errantly includes a triangle of the same size but rotated $180^\circ$, to the right of $\lambda=-m$ (see Figure \ref{fig_n9_unstable_grid}).

\begin{figure} \label{fig_n9_unstable_grid}
\begin{overpic}[scale=1.0]{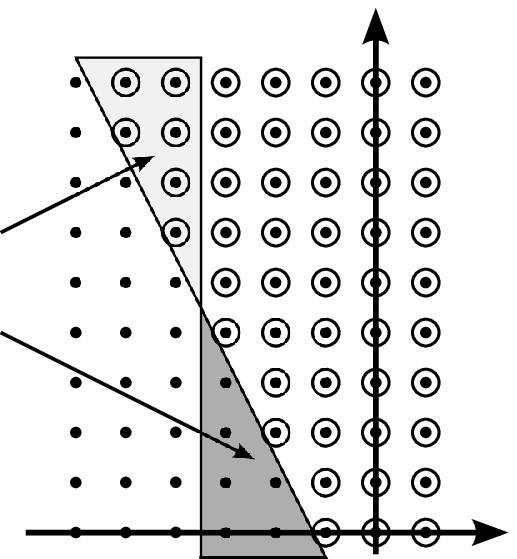}
\put(-10,57){$T_n$}
\put(-10,40){$T_n'$}
\put(84,10){$\lambda$}
\put(75,93){$e$}
\end{overpic}
\caption{The triangles $T_n$ and $T'_n$.}
\end{figure}

Let $T_n$ denote the triangle of lattice points in $Q_n$ with $\lambda<-m$. Each lattice point $(\lambda,e)$ in $Q_n$ contributes a sum of terms $t^{2d_{\lambda,\ubar e}}$ consisting of vectors $\ubar e$ with $\sum_ke_k=e$. There are $\binom{n}{e}$ such vectors, so the contribution of the point $(\lambda,e)$ is $\binom{n}{e}t^{4\lambda+2n-2+2e}$. Each point $(\lambda,e)$ in $T_n$ pairs with the point $(\lambda',e')=(-n-\lambda,n-e)$ obtained by rotating $180^\circ$ about the center $(-\frac{n}{2},\frac{n}{2})$. The approximating half space $\lambda\geq -m$ misses $(\lambda,e)\in T_n$, but wrongfully includes $(\lambda',e')$ in the rotated triangle $T_n'$. The correction for this in the sum over unstable strata is therefore
\begin{equation} \label{correction_term}
\binom{n}{e}\left(t^{n-2+(4\lambda+n+2e)}-t^{n-2-(4\lambda+n+2e)}\right).
\end{equation}
We want to separate the sum $\sum_{Q_n} t^{2d_{\lambda,\ubar e}}$ into the part $\sum_{\lambda\geq-m,\ubar e}t^{2d_{\lambda,\ubar e}}$, plus the correction which we momentarily denote by $\sum_{T_n}(\cdots)$, consisting of a sum of terms (\ref{correction_term}). We focus first on the the rectangular sum. We have:
\begin{align*}
\sum_{\lambda\geq-m,\ubar e}t^{2d_{\lambda,\ubar e}}=&\;t^{2n-2}\sum_{\lambda\geq -m}t^{4\lambda}\left(\sum_{e_1=0,1}t^{2e_1}\cdots\sum_{e_n=0,1}t^{2e_n}\right) \\
=&\;(1-t^4)^{-1}(1+t^2)^n
\end{align*}
Going back to the formula for the Poincar\'e polynomial plugging in $g=0$ we get:
\begin{align*}
P_t(\mathcal{R}_{0,n})=&\;\frac{1}{(1-t^2)^2}\left((1+t^2)^{n-1}-(1-t^2)\left[(1-t^4)^{-1}(1+t^2)^n+\sum_{T_n}t^{2d_{\lambda,\ubar e}}\right]\right) \\
=&\;-\frac{1}{1-t^2}\sum_{T_n}(t^{2d_{\lambda,\ubar e}})
\end{align*}
We see that the entire cohomology is determined by the contributions from the triangle $T_n$. Let $T_n'$ denote the rotation of $T_n$ about the center point $(-\frac{n}{2},\frac{n}{2})$. Pairs $(\lambda,e)\in T_n'$, that is with $4\lambda+2e<-n$ and $\lambda\geq -m$, are in bijection with pairs $(\mu,e)$ with $4\mu+2e\leq n-3$ and $\mu\geq 0$. The contribution to the above sum for such a pair is obtained by dividing equation (\ref{correction_term}) by $t^2-1$:
\begin{equation} \label{correction_term_II}
\binom{n}{e}\sum_{i=0}^{D} t^{4\mu+2e+2i}
\end{equation}
where $D=n-3-(4\mu+2e)$. Now, we consider the graded algebra
$$\mathbb{B}_{0,n}:=\CC[\alpha,\beta,\delta_1,\ldots,\delta_n]/(\delta_i^2),$$
where $\alpha$, $\beta$, $\delta_k$ have degrees 2, 4, and 2, respectively. The dimenension in each grading of $\mathbb{B}_{0,n}$ can be found by counting monomials of the form $\alpha^a\beta^b\delta_1^{e_1}\cdots\delta_n^{e_n}$ for $e_k=0,1$. For fixed $e=\sum_ke_k$, the number of such monomials with fixed $a$ and $b$ is $\binom{n}{e}$. The contribution to $P_t(\mathcal{R}_{0,n})$ from such a monomial can be viewed as coming from the term $t^{2a+4b+2e}$ in the above sum (\ref{correction_term_II}) for $\mu=b$ and $i=a$. This shows that the Poincar\'e polynomials of $\mathcal{R}_{0,n}$ and $\mathbb{B}_{0,n}$ are equal up to the middle dimension $n-3$.
\end{proof}

We can use this theorem to get our first glimpses of the recursive nature of the $\mathcal{R}_{g,n}$ in the number of parabolic points $n$, and eventually the genus $g$.

\begin{cor} \label{cor_genus_0_poincare_poly}
Suppose $n=2m+1>3$. For the genus 0 representation varieties, we have:
\begin{equation} \label{genus_0_poincare_poly}
P_t(\mathcal{R}_{0,n+2})=(1+t^2)^2P_t(\mathcal{R}_{0,n})+2^{n-1}t^{n-1}
\end{equation}
\end{cor}
\begin{proof}
This follows from the isomorphism
$$\mathbb{B}_{0,n+2}\cong\mathbb{B}_{0,n}\otimes\CC[\delta_{n+1},\delta_{n+2}]/(\delta_{n+1}^2,\delta_{n+2}^2).$$
Let $\mathcal{T}_{0,n}^d$ denote the collection of monomials $\alpha^a\beta^b\delta_1^{e_1}\cdot\ldots\cdot\delta_n^{e_n}$ in $\mathbb{B}_{0,n}$ of degree $d$. For example, $$T_{0,n}^4=\{\alpha^2,\beta,\alpha\delta_1,\ldots,\alpha\delta_n,\delta_1\delta_2,\ldots,\delta_{n-1}\delta_n\}$$
For $d$ up to the middle dimension of $\mathcal{R}_{0,n}$, the coefficient of $t^d$ in $P_t(\mathcal{R}_{0,n})$ is the cardinality of $T_{0,n}^d$. For $d$ strictly less than the middle dimension $2m$ for $\mathcal{R}_{0,n+2}$, any monomial in $T_{0,n+2}^d$ can be obtained by multiplying a monomial in $T_{0,n}^{d'}$ by $\delta_{n+1}$, $\delta_{n+2}$, both, or neither, for $d'\leq 2m-2$, which is the middle dimension for $\mathcal{R}_{0,n}$. This argument shows that $P_t(\mathcal{R}_{0,n+2})$ and $(1+t^2)^2P_t(\mathcal{R}_{0,n})$ agree up to degree $2m-2$.

In degree $2m$, we consider the collection of those monomials in $T_{0,n+2}^{2m}$ arising from multiplying a monomial in $T_{0,n}^{2m-2e_1-2e_2}$ by $\delta_{n+1}^{e_1}\delta_{n+2}^{e_2}$ for $(e_1,e_2)=(1,0),(0,1),(1,1)$, along with those arising from multiplying a monomial in $T_{0,n}^{2m-4}$ by $\beta$. These 4 disjoint subsets contribute $2\abs{T_{0,n}^{2m-4}}+2\abs{T_{0,n}^{2m-2}}$ to $\abs{T_{0,n+2}^{2m}}$, which is the same as the contribution to the coefficient of $t^{2m}$ in $P_t(\mathcal{R}_{0,n+2})$ from $(1+t^2)^2P_t(\mathcal{R}_{0,n})$ by Poincar\'e duality. It remains to count those monomials in $T_{0,n+2}^{2m}$ not arising this way. These are precisely those monomials in only the variables $\alpha$ and $\delta_i$'s for $i\leq n$. Of course, these are in bijection with vectors $\ubar e=(e_1,\ldots,e_n)$ with $e=\sum_ke_k\leq m$, which is easy seen to equal $2^{n-1}=2^{2m}$. The equation (\ref{genus_0_poincare_poly}) now follows immediately.
\end{proof}

With a bit more work, we get:

\begin{cor} \label{cor_genus_g_poincare_poly}
We have the following recursive formulas for the the Poincar\'e polynomials, where $n\geq 1$:
\begin{align}
P_t(\mathcal{R}_{g,n+2})=&\;(1+t^2)^2P_t(\mathcal{R}_{g,n})+2^{n-1}t^{2g+n-1}(1+t)^{2g} \label{genus_g_p_poly_rec_n} \\
P_t(\mathcal{R}_{g+1,n})=&\;(1+t^3)^2P_t(\mathcal{R}_{g,n})+2^{n-1}t^{2g+n-1}(1+t)^{2g}(1+t^2) \label{genus_g_p_poly_rec_g}
\end{align}
\end{cor}
\begin{proof}
Let $S_{g,n}$ denote the sum $\sum_{\lambda,\ubar e}t^{2d_{\lambda,\ubar e}}$ in (\ref{poincare_poly}). We can solve (\ref{genus_0_poincare_poly}) for an equation giving $S_{0,n+2}$ in terms of $S_{0,n}$:
\begin{equation} \label{strata_sum}
S_{0,n+2}=(1+t^2)^2S_{0,n}-2^{n-1}(1-t^2)t^{n-1}
\end{equation}
Increasing the genus does not change the domain of the sum in (\ref{poincare_poly}), but multiplies it by the factor $t^{2g}$ so that $S_{g,n}=t^{2g}S_{0,n}$. Hence:
\begin{equation} \label{strata_sum_g}
S_{g,n+2}=(1+t^2)^2S_{g,n}-2^{n-1}(1-t^2)t^{n-1}
\end{equation}
holds for all $g$. Plugging this equation into (\ref{poincare_poly}) for arbitrary $g$, the equation (\ref{genus_g_p_poly_rec_n}) follows by some straightforward algebra.
To prove formula (\ref{genus_g_p_poly_rec_g}), we may argue by induction on $n$. Let $\Delta_{g,n}$ denote the difference $P_t(\mathcal{R}_{g+1,n})-(1+t^3)^2P_t(\mathcal{R}_{g,n})$. Suppose that $\Delta_{g,n}=2^{2g+n-1}t^{n-1}(1+t)^{2g}(1+t^2)$. Now, we use (\ref{genus_g_p_poly_rec_n}) to compute $\Delta_{g,n+2}$. We have:
\begin{align*}
\Delta_{g,n+2}=&\;P_t(\mathcal{R}_{g+1,n+2})-(1+t^3)^2P_t(\mathcal{R}_{g,n+2}) \\
			  =&\;(1+t^2)^2\Delta_{g,n}+2^{n-1}t^{2g+n+1}(1+t)^{2g+2}-2^{n-1}t^{2g+n-1}(1+t)^{2g}(1+t^3)^2 \\
			  =&\;2^{n-1}t^{2g+n-1}(1+t)^{2g}\left[(1+t^2)^3+t^2(1+t)^2-(1+t^3)^2\right] \\
			  =&\;2^{n-1}t^{2g+n-1}(1+t)^{2g}\left(4t^2+4t^4\right)=2^{n+1}t^{2g+n+1}(1+t)^{2g}(1+t^2)
\end{align*}
which is exactly what need for $(\ref{genus_g_p_poly_rec_n})$ to hold in general. The base cases of $n=1$ are easy to check and left to the reader.
\end{proof}

\begin{remark}
The formula (\ref{genus_g_p_poly_rec_n}) should be interpreted as saying that the homology of $\mathcal{R}_{g,n+2}$ is the same as that of $\CP^1\times\CP^1\times\mathcal{R}_{g,n}$, with $2^{n-1}$ times the homology of the Jacobian $J(\Sigma)$ added in the middle dimensions. On the other hand, formula (\ref{genus_g_p_poly_rec_g}) says that the homology of $\mathcal{R}_{g+1,n}$ is obtained as the homology of $\SU(2)\times\SU(2)\times\mathcal{R}_{g,n}$ with $2^{n-1}$ times the homology of the genus $g$ Jacbobian crossed with $\mathbb{P}^1$ added in the middle dimensions. Corollary \ref{cor_genus_0_poincare_poly} is strong evidence that the cohomology \emph{ring} is a free algebra on generators $\alpha$, $\beta$, and $\delta_i$, modulo $\delta_i^2=x$ for some $x$, up to the middle dimension, with relations in degree 2 higher than the middle.
\end{remark}

Since both of the formula in Corollary \ref{cor_genus_g_poincare_poly} are essentially geometric series recursions, they can in fact be solved for explicitly. Performing this computation gives:

\begin{thm} \label{thm_poincare_poly_closed}
The Poincar\'e polynomial of $\mathcal{R}_{g,n}$ is given by:
\begin{equation} \label{closed_form_poly}
P_t(\mathcal{R}_{g,n})=\frac{(1+t^2)^n(1+t^3)^{2g}-2^{n-1}t^{2g+n-1}(1+t)^{2g}(1+t^2)}{(1-t^2)(1-t^4)}
\end{equation}
\end{thm}

\begin{cor} \label{cor_rank_degree_2}
$H^2(\mathcal{R}_{g,n})=n+1$
\end{cor}

\begin{example}
If we set $g=0$ but let $n=5$, we expect a 4-dimensional moduli space. By Corollary \ref{cor_genus_0_poincare_poly}:
$$P_t(\mathcal{R}_{0,5})=(1+t^2)^2P_t(\mathcal{R}_{0,3})+4t^2=1+6t^2+t^4$$
(or more simply we could just use Corollary \ref{cor_rank_degree_2}). As noticed by Boden in \cite{boden_parabolic}, the paper \cite{kirk_klassen} proves that the only possibility for a four-dimensional representation variety (in genus 0) with $b_2=6$ is $\CP^2\#5\obar{\CP}^2$. This follows by considering the structure of a toric variety on these moduli spaces.
\end{example}
%
%
%
%

\section{Canonical Line Bundles on $\mathcal{R}_{g,n}$} \label{section_line_bundles}

In this section we describe a natural collection of line bundles over $\mathcal{R}_{g,n}$ as well as explicit submanifolds of the moduli space whose Poincar\'e duals can be used to write down the first Chern classes of these line bundles. These line bundles and submanifolds are exactly of the type studied in \cite{weitsman_cohom}, though the assumptions there on the genericity of parabolic weights and methods of symplectic reduction preclude the full application of his results to our case. Specifically, setting all the weights to be $1/4$ makes the calculation of the Poincar\'e duals of the first Chern classes of the line bundles somewhat more tedious, and the inductive technique in that paper of reducing the number of parabolic points by one fails, because we require $n$ to be odd. However, we give an analogous scheme by which $n$ can be reduced by two. The resulting recursive description of $\mathcal{R}_{g,n}$ will eventually allow us to completely write down the cohomology ring for $g=0$, using these Chern classes as generators. We first define these line bundles, and show that their first Chern classes are naturally identified with a subset of the generating set of classes found in \S\ref{subsection_univ_bundles}.

\subsection{Line bundles $V_k$ and submanifolds $D^\pm_{k,l}$}

The line bundles we want to study are those associated to explicit $\U(1)$ bundles over $\mathcal{R}_{g,n}$. Recall our standard generating set $\{a_j,d_k,\zeta\}$ (for $j=1,\ldots,2g$ and $k=1,\ldots,n$) with $\zeta^2=1$ for the central extension $\wh\Gamma$ of the fundamental group of $\Sigma^*$. The unreduced representation space $\wt{\mathcal{R}}$ consists of all maps $\wh\Gamma\to\SU(2)$, with a trace condition on the $d_k$ and with $\zeta$ mapping to $-1$. Viewing $\wt{\mathcal{R}}_{g,n}$ as sitting in $\SU(2)^{2g+n}$, inside of $\wt{\mathcal{R}}_{g,n}$ we have, for each $k$, a subspace
$$V_k:=\{(S_1,\ldots,S_{2g},T_1,\ldots,T_n)\in\wt{\mathcal{R}}_{g,n}|\text{ }T_k=\mathbf{i}\}$$
where the $k$ parabolic coordinate is exactly $\mathbf{i}$. Any representation in $\wt{\mathcal{R}}_{g,n}$ can be conjugated to one in $V_k$, so the quotient map $V_k\to\mathcal{R}_{g,n}$ is surjective. A fiber can be identified with the stabilizer of $\mathbf{i}$: if $\rho\in V_k$ maps to $\obar\rho$ in $\mathcal{R}_{g,n}$, any other element of the preimage of $\obar\rho$  arises by conjugating $\rho$, but to stay in $V_k$ the conjugating element $g\in\SU(2)$ must fix $\mathbf{i}$. This stabilizer is the one parameter subgroup $S^1_\mathbf{i}$ through $\mathbf{i}$. The action of $S^1_\mathbf{i}$ on the fiber decends to a free one of the quotient $S^1_\mathbf{i}/\{\pm 1\}$; both are isomorphic to the group $\U(1)$, so
$$V_k\to\mathcal{R}_{g,n}$$
becomes a principal $\U(1)$-bundle.

For each pair $k,l$, let we define the subspace $D^\pm_{k,l}\subset\mathcal{R}_{g,n}$:
$$D^\pm_{k,l}=\{[\ubar S,\ubar T]\;|\;T_k=\pm T_l\}.$$
The union $D^+_{k,l}\cup D^-_{k,l}$ is exactly the locus where $T_k$ and $T_l$ commute. It is not difficult to see that the spaces $D^\pm_{k,l}$ are smooth, orientable, connected, real codimension 2 submanifolds of $\mathcal{R}_{g,n}$. It is a careful study of these spaces and their intersections which gives us a great deal of information about the cohomology ring of $\mathcal{R}_{g,n}$.

The first step, as seen in \cite{weitsman_cohom}, is to notice the line bundles $V_k$ and the submanifolds $D_{k,l}^\pm$ are intimately related.
\begin{lemma} \label{V_k_trivialization}
The $\U(1)$ bundle $V_k$ is trivial on the complement of $D^+_{k,l}\sqcup D^-_{k,l}$ in $\mathcal{R}_{g,n}$, for any $l\neq k$.
\end{lemma}
\begin{proof}
We shall describe an explicit section $\mathcal{R}_{g,n}\to V_k$ outside $D^+_{k,l}\sqcup D^-_{k,l}$. Let $[\ubar S,\ubar T]\in\mathcal{R}_{g,n}$ be such that $T_k$ and $T_l$ do not commute. It is not hard to see that there is a unique representative of $[\ubar S,\ubar T]$ for which $T_k=\mathbf{i}$, the $\mathbf{k}$-component of $T_l$ is 0, and the $\mathbf{j}$-component of $T_l$ is positive, assuming $T_k$ and $T_l$ are not equal or antipodal. Perhaps the best intuition for this comes from geography: there is a unique oriented orthogonal transformation of the earth which maps a given point to the north pole and any other point not the south pole to lie on the Prime Meridian. This recipe gives a unique representative $\ubar T$ for any given conjugacy class, and this assignment is certainly continuous in $\mathcal{R}_{g,n}$, so gives a section.
\end{proof}

\begin{cor} \label{V_k_poincare_dual}
Suppose $n\geq 3$, and let $[D^\pm_{k,l}]$ denote the homology classes associated to $D^\pm_{k,l}$ with some choice of orientations. Then there are integers $r,s$ such that
\begin{equation}
PD(c_1(V_k))=r[D^+_{k,l}]+s[D^-_{k,l}]
\end{equation}
\end{cor}

This raises the question of what happens when $n=1$. The author can find no obvious projection $\mathcal{R}_{g,1}\to\mathcal{R}_{g,0}$ by studying the representation varieties. Nonetheless, such a projection map is easily defined by studying the correpsonding moduli spaces of stable parabolic bundles. Suppose $\mathcal{E}$ is a stable parabolic bundle over $\Sigma$ with one parabolic point $x_1$. Assuming $\deg\mathcal{E}=-1$ (which we are free to do by tensoring with a line bundle of degree $-1$), we have $\deg^\PAR\mathcal{E}=0$. We claim that $\mathcal{E}$ is actually stable as a nonparabolic bundle. Indeed, if $\mathcal{L}$ were a (nonparabolic) destabilizing line bundle, then we must have $\deg\mathcal{L}>-\tfrac{1}{2}$ and so $\deg\mathcal{L}\geq 0$. But then $\deg^\PAR\mathcal{L}\geq\tfrac{1}{4}$, so $\mathcal{L}$ would be a parabolic destabilizing line bundle for the parabolic bundle $\mathcal{E}$, which is a contradiction. Hence, we have a forgetful map
$$f:\mathcal{M}^0_{g,1}\to\mathcal{M}^0_{g,0}$$
It is a surjective morphism of projective algebraic varieties, whose fiber is the $\mathbb{P}^1$ coming from the choice of parabolic line $F_1\subset\mathcal{E}_{x_1}$.

In fact what this shows is that when $n=1$, $\Cs\subset\mathcal{C}$ actually parametrizes stable nonparabolic holomorphic structure on $E\to\Sigma$ as well. To get $\mathcal{M}_{g,0}$, we simply take the quotient by the larger (nonparabolic) gauge group $\GC$. The parabolic gauge group $\GCp$ is a subgroup of $\GC$ and the coset space $\GC/\GCp$ is a $\mathbb{P}^1$. After passing to the fixed determinant subspaces, this realizes the $\mathbb{P}^1$ bundle
$$\mathcal{M}^0_{g,1}=\Cs/\GCp\to\Cs/\GC=\mathcal{M}^0_{g,0}.$$
Let us study the corresponding universal bundles. Over $\mathcal{M}^0_{g,1}\times\Sigma$ we have $\mathbb{P}_{g,1}$ and section $\mathbf{s}_k$ over $x_1\in\Sigma$, constructed as a quotient of $\mathbb{P}(\pi^*E)$ over $\Cs\times\Sigma$. Over $\mathcal{M}^0_{g,0}$ we have $\mathbb{P}_{g,0}$. There is an obvious fiberwise map $\mathbb{P}_{g,1}\to\mathbb{P}_{g,0}$ covering the bundle map $\mathcal{M}^0_{g,1}\times\Sigma\to\mathcal{M}^0_{g,0}\times\Sigma$ which explicitly realizes the first as isomorphic to the pullback of the second. It is straightforward to check also that for the $\SO(3)$ universal bundle $\mathbf{E}^\ad(g,n)$ over $\mathcal{R}_{g,n}\times\Sigma$, we have
$$\mathbf{E}^\ad_{g,1}\cong f^*(\mathbf{E}^\ad_{g,0}).$$

Before tackling the computation of the constants $r,s$ in Corollary \ref{V_k_poincare_dual}, we first relate the line bundles $V_k$ to the universal bundle constructions from \S\ref{subsection_univ_bundles} and make a few notes on the symmetry inherent in our setup.

\subsection{The Universal Bundle and $V_k$'s} \label{subsection_univ_bundles_and_V_ks}

We would like to identify $V_k$ with natural bundles arising from the universal pair over $\mathcal{R}_{g,n}\times\Sigma$. We recall the universal bundle pair $(\mathbf{E}^\ad,\{\mathbf{V}_k\})$ over $\mathcal{R}_{g,n}\times\Sigma$, constructed as the quotient of a pullback pair $(\pi^*\Ad E,\{\pi^*H_k\})$ by the gauge group $\mathcal{G}$. Fix a basepoint $z\in\Sigma^*$ away from the punctures, and consider the subgroup $$\mathcal{G}_z\subset\mathcal{G},$$
called the ``based'' gauge group, of gauge transformations which are the identity at $z$. The quotient $\mathcal{G}/\mathcal{G}_z$ is isomorphic to $\SU(2)$, and $\mathcal{A}_\text{flat}/\mathcal{G}_z$ is naturally an $\SO(3)$ bundle over $\mathcal{R}_{g,n}=\mathcal{A}_\text{flat}/\mathcal{G}$. Given a connection $A$ on $\Ad E$, we write $[A]$ for the $\mathcal{G}$ equivalence class, and $\langle A\rangle_z$ for the equivalence class modulo $\mathcal{G}_z$. 
\begin{lemma} \label{partial_quotient}
The residual $\SO(3)$ action on the quotient $\mathcal{A}_\text{flat}/\mathcal{G}_z$ is isomorphic as a principal bundle to $\wt{\mathcal{R}}_{g,n}\to\mathcal{R}_{g,n}$.
\end{lemma}
\begin{proof}
We define an explicit map
$$\mathcal{A}_\text{flat}/\mathcal{G}_z\to\wt{\mathcal{R}}_{g,n}$$
using the holonomy of flat connections. The point is that by modding out by the based gauge group, we can still get from a flat connection its true holonomy representation and not just its conjugacy class. To wit, recall we have fixed an auxillary point $z_0$ near $z$ and trivialization of $E$ outside $z_0$, so that a flat connection $A$ on $\Ad E$ gives rise to a flat $\SU(2)$ connection $\su(A)$ away from $z_0$. This gives a homomorphism
$$\text{hol}_z(\su(A)):\wh\Gamma\to\SU(2).$$
The based gauge group preserves the holonomy up to conjugation by the local group $(\mathcal{G}_z)|_z$, which is trivial by definition, and so $\text{hol}_z(\su(A))$ is independent of the based gauge representative. We therefore get a map $\langle A\rangle_z\mapsto\hol_z(\su(A))$ which is a fibered bijection between the bundles, and clearly $\SO(3)$-equivariant.
\end{proof}
\begin{lemma} \label{equiv_bundles_over_R}
Let $W$ be the rank three $\RR$ vector bundle associated to the principal $\SO(3)$ bundle $\wt{\mathcal{R}}_{g,n}\to\mathcal{R}_{g,n}$. Then for any $z\in\Sigma^*$ away from the punctures, $W$ is isomorphic to $\mathbf{E}^\ad|_{\mathcal{R}_{g,n}\times\{z\}}$.
\end{lemma}
\begin{proof}
By definition, $W$ is the bundle $\wt{\mathcal{R}}_{g,n}\times_{\SO(3)}\so(3)$, and a vector in the total space of the restriction of $\mathbf{E}^\ad$ to the $\mathcal{R}_{g,n}$ slice through $z$ is the $\mathcal{G}$ orbit of a pair $(t,A)$ where $t\in(\Ad E)_z$. A map $\mathbf{E}^\ad\to W$ over the $\mathcal{R}_{g,n}$ slice can be defined by
$$\mathcal{G}\cdot(t,A)\mapsto\obar{(\text{hol}_z(\su(A)),t)}$$
where $t$ is an element of $\Ad E_z\cong\so(3)$. Independence on the choice of representative $(t,A)$ is seen by a straightforward unwinding of the definitions, which we carry out for completeness. For $g\in\mathcal{G}$, we must show that $(\text{hol}_z(\su(g\cdot A)\rangle),g_z\circ t\circ g_z^{-1})$ and $(\text{hol}_z(\langle A\rangle_z),t)$ are $\SO(3)$-equivalent. Indeed: the holonomy of $\su(g\cdot A)$ is that of $\su(A)$ conjugated by $g_z$, so these two pairs are identified by the action of the element $\obar g_z\in\SO(3)$. The map is certainly a linear isomorphism on fibers.
\end{proof}

When the basepoint $z$ is near an $x_k$ (say $s_k(z)<3/4$ for some $k$), we get a different kind of based gauge group which we shall denote by $\mathcal{G}_k$, consisting of gauge transformations which are the identity near the $k$th puncture. Since the original gauge group $\mathcal{G}$ was only allowed to take values $\U(1)\subset\SU(2)$ near the punctures, the quotient $\mathcal{G}/\mathcal{G}_k$ is isomorphic to $\U(1)$. Hence, the quotient $\mathcal{A}_\text{flat}/\mathcal{G}_k$ is a principal $\U(1)$ bundle over $\mathcal{R}_{g,n}$. Let $\langle A\rangle_k$ denote the $\mathcal{G}_k$ equivalence class of $A$. We now finally bring the discussion back to our line bundles $V_k$ defined at the beginning of the section. 
\begin{lemma} \label{partial_U1_quotient}
The $\U(1)$ bundle $\mathcal{A}_\text{flat}/\mathcal{G}_k\to\mathcal{R}_{g,n}$ is isomorphic to the bundle $V_k\to\mathcal{R}_{g,n}$.
\end{lemma}
\begin{proof}
The proof is analogous to that of Lemma \ref{partial_quotient}. Since $x$ is near the puncture, we have a chosen unitary identification $E_z\to\CC^2$ under which the line $F_k$ is spanned by $(1,0)$. Given a connection $A$, the holonomy $\text{hol}_z(\su(A))$ gives an element of $\wt{\mathcal{R}}_{g,n}$. Since $E_z$ is fixed by the based gauge group, this element is independent of the representative for the orbit $\mathcal{G}_k\cdot A$. Moreovoer, because we have fixed the 1-form of $A$ (and thus of $\su(A)$) near the punctures, the holonomy of $\su(A)$ around a small loop around $x_k$ staying inside the locus where $t\geq 1/2$ will be precisely $\mathbf{i}$. This implies that $\hol_z(\su(A))\in\wt{\mathcal{R}}_{g,n}$ is in $V_k$. The identification $\langle A\rangle_k\mapsto\hol_z(\su(A))$ is manifestly a $\U(1)$-equivariant bijection.
\end{proof}
\begin{lemma} \label{line_bundles_over_R} 
There is an isomorphism
$$\mathbf{V}_k\cong V_k$$
\end{lemma}
\begin{proof}
We mimic the proof of Lemma \ref{equiv_bundles_over_R}. For clarity of notation, we let $L_k$ specifically denote the complex line bundle $V_k\times_{\U(1)}\CC$ associated to the principal $\U(1)$-bundle $V_k\subset\wt{\mathcal{R}}_{g,n}$ over $\mathcal{R}_{g,n}$. Here, $w\in\U(1)\subset\CC$ acts on $\CC$ by simple multiplication, but the action on $V_k$ (in order to make it a free one) is by conjugating a representation by a choice of square root $\diag(w^{1/2},w^{-1/2})$. Now, a vector in the total space of $\mathbf{V}_k$ is a gauge orbit of a pair $(t,A)$ where $t\in H_k\subset(\Ad E)_{x_k}\subset\End(E_{x_k})$. With our fixed isomorphism $E_{x_k}\cong\CC^2$, $t$ is just a matrix $\left(\begin{smallmatrix}0&w\\-\obar w&0\end{smallmatrix}\right)$. Still letting $z$ be a basepoint near $x_k$, we define a map $\mathbf{V}_k\to L_k$ via:
$$\mathcal{G}\cdot\left(\left(\begin{smallmatrix}0&w\\-\obar w&0\end{smallmatrix}\right),A\right)\mapsto\obar{(\hol_z(\su(A)),w)}$$
Let us prove the independence on gauge representative. Suppose $g\in\mathcal{G}$, and let $g_z=(\begin{smallmatrix}v&0\\0&v^{-1}\end{smallmatrix})$ for a unit length $v\in\CC$. Then we have:
\begin{align*}
&\hol_z\left(\su(g\cdot A)\right)= \left(\begin{smallmatrix}v&0\\0&v^{-1}\end{smallmatrix}\right)\cdot\hol_z(\su(A))\cdot
\left(\begin{smallmatrix}v&0\\0&v^{-1}\end{smallmatrix}\right)^{-1}\\
\text{and: }& \left(\begin{smallmatrix}v&0\\0&v^{-1}\end{smallmatrix}\right)\left(\begin{smallmatrix}0&w\\-\obar w&0\end{smallmatrix}\right)\left(\begin{smallmatrix}v&0\\0&v^{-1}\end{smallmatrix}\right)^{-1}=
\left(\begin{smallmatrix}0&v^2w\\-v^{-2}\obar w&0\end{smallmatrix}\right).
\end{align*}
The pair
$$\left(\left(\begin{smallmatrix}v&0\\0&v^{-1}\end{smallmatrix}\right)\cdot\hol_z(\su(A))\cdot
\left(\begin{smallmatrix}v&0\\0&v^{-1}\end{smallmatrix}\right)^{-1},v^2w\right)$$
is equivalent to $(\hol_z(\su(A)),b)$ under the action of $v^2\in\U(1)$. This shows independence of the representative, and the resulting map is certainly a linear isomorphism.
\end{proof}

\begin{cor}
We have $c_1(V_k)=c_1(\mathbf{V}_k)$ in $H^2(\mathcal{R}_{g,n};\QQ)$.
\end{cor}

The upshot of this is that we may now use the $D^\pm_{k,l}$'s to study the cohomology ring of $\mathcal{R}_{g,n}$: we know that their first Chern classes of the $V_k$'s are Poincar\'e dual to a pair of $D^\pm_{k,l}$'s, and we also know that these classes are part of a natural generating set for the cohomology. What remains is to determine the relations they satisfy.

\subsection{Action of the Mapping Class Group and Flips}

When analyzing the $\mathcal{D}^\pm_{k,l}$'s, it will be convenient to exploit the considerable symmetry in $\mathcal{R}_{g,n}$ arising from the symmetries of $\Sigma^*$.

\vspace{12pt}

\noindent\textbf{Action of the Mapping Class Group.} On the space of connections on $E\to\Sigma^*$ there is a (left) action of the basepoint preserving mapping class group $\Mod^*_{g,n}$ by the operation of pullback. Let us make this precise. Fix a basepoint $z\in\Sigma$ and for each $\phi\in\Mod_{g,n}$, fix representative $f:\Sigma^*\to\Sigma^*$ for $\phi$ with $f(z)=z$ and with the property that on the coordinate cylinders $U_k\cong(0,1)\times S^1$ near each puncture, $f$ agrees with the coordinates. In other words, if $f$ maps the $k$th puncture to the $j$th, then it gives a diffeomorphism $U_k$ to $U_j$ which is just $(s_k,\theta_k)\mapsto(s_j,\theta_j+\xi)$ for some constant $\xi$. In addition, fix a fiberwise isometry $F:E\to E$ covering $f$ which is the identity at $z$, and let $\wh F:\Ad E\to\Ad E$ denote the induced isometry. Given a connection $A\in\mathcal{A}$, we define a new connection $\wt M_\phi(A)$ by setting, for a local section $t$ of $\Ad E$ and tangent vector $X\in T_x\Sigma^*$:
\begin{equation} \label{mapping_class_conn}
d_{\wt M_\phi(A)}(t)\cdot X=\wh F(d_A(\wh F\vphantom{F}^{-1}\circ t\circ f)\cdot df(X))
\end{equation}
This complicated looking formula is exactly what is required so that a section $t$ is $M_\phi(A)$-flat if and only if $\wh F\circ t\circ f^{-1}$ is $A$-flat. We see that $\wt M_\phi$ gives a map on $\mathcal{A}$ which preserves the set of flat connections. Now, if $g$ is a gauge transformation, then $t$ is $\wt M_\phi(g\cdot A)$-flat if and only if $\Ad g\circ \wh F\circ t\circ f^{-1}$ is $A$-flat, if and only if $(\wh F\vphantom{F}^{-1}\circ \Ad g\circ \wh F)\circ t$ is $\wt M_\phi(A)$-flat. Moreover, it is clear that
$$(\wh F\vphantom{F}^{-1}\circ \Ad g\circ \wh F)=\Ad(F^{-1}\circ g\circ F)$$
and thus
$$\wt M_\phi(g\cdot A)=(F^{-1}\circ g\circ F)\cdot \wt M_\phi(A)$$
This shows that $\wt M_\phi$ preserves gauge equivalence classes. Hence, $\wt M_\phi$ descends to a diffeomorphism $M_\phi$ on $\mathcal{R}_{g,n}$. One can check that although the map $\wt M_\phi$ on $\mathcal{A}$ itself may depend very much on the choices made (for example, $f$ and $F$), these choices do not affect the holonomy of a flat connection and so $M_\phi$ depends only the mapping class $\phi$. Indeed, the holonomy of $M_\phi(A)$ around a loop $\gamma$ based at $z$ is clearly just the holonomy of $A$ around the loop $f(\gamma)$, which is certainly invariant under the choice of $f$ for $\phi$ and completely independent of $F$. We therefore get an action of the mapping class group $\Mod^*_{g,n}$ on $\mathcal{R}_{g,n}$. Another way to see this action is simply by looking at the action of $\Mod^*_{g,n}$ on the $\ZZ/2$ extension $\wh\Gamma$ of $\pi_1(\Sigma^*)$: precomposing representations by this action induces an one on $\wt{\mathcal{R}}_{g,n}$ (which lifts the action from $\mathcal{R}_{g,n}$).

Inside of $\Mod^*_{g,n}$, there is a copy of the braid group $B_n$ on $n$ strands arising from those diffeomorphisms which are supported on a small disk containing the punctures. There is a natural surjective homomorphism $\tau:B_n\to S_n$ where $S_n$ is the symmetric group on $n$ elements given by observing how a mapping class permutes the punctures.

\begin{lemma} \label{mod_gn_action_on_V_k}
Let $\phi\in B_n\subset\Mod_{g,n}$, and set $\sigma=\tau(\phi)$. Then $M_\phi(D_{\sigma(k),\sigma(l)}^\pm)=D_{k,l}^\pm$, and $M_\phi^*(V_{\sigma(k)})\cong V_k$.
\end{lemma}
\begin{proof}
The first statement is straightforward. If $\phi$ is a mapping class carrying neighborhoods of $x_k$ and $x_l$ to those of $x_{\sigma(k)}$ and $x_{\sigma(l)}$ respectively, then the holonomies of the connection $M_\phi(A)$ around $x_k$ and $x_l$ are just those of $A$ around $x_{\sigma(k)}$ and $x_{\sigma(l)}$. Hence, $M_\phi$ carries connections with equal or antipodal holonomies around $x_{\sigma(k)}$ and $x_{\sigma(l)}$ to those with equal or antipodal connections around $x_k$ to $x_l$, respectively.

For the second statement about $V_k$, let us make use of an explicit presentation for $B_n$. It is well known that $B_n$ is generated by $n-1$ elementary braids $b_1,\ldots,b_{n-1}$ subject to the relation $b_ib_{i+1}b_i=b_{i+1}b_ib_{i+1}$ for each $i$, and where $b_i$ and $b_j$ commute for $\abs{i-j}\geq 2$. The action of $B_n$ on $\pi_1(\Sigma^*)$ is given by
\begin{align*}
b_i(d_k)=\begin{cases} d_k\text{, }&\text{if }k\neq i,i+1 \\
														d_{i+1}\text{, }&\text{if }k=i \\
														d_{i+1}^{-1}d_id_{i+1}\text{, }&\text{if }k=i+1 \end{cases},
\end{align*}
and $B_n$ fixes the $a_j$'s. The lemma will be proved if we can show it is true when $\phi$ arises from an elementary braid $b_i$, since these generate the action. Since $\tau(b_i)$ is just the transposition $(i\;\;i+1)$, we need to show that $M_\phi^*(V_{i+1})$ and $V_i$ are isomorphic. Clearly, the action of $M_\phi$ on $\wt{\mathcal{R}}_{g,n}$ maps the subset $V_{i+1}$ to $V_{i}$ (since $\rho\circ b_i(d_i)=\rho(d_{i+1})=\mathbf{i}$). In fact, it is easy to check that this map is equivariant with respect to the action of $\U(1)$ by conjugation. This proves directly that the pullback of $V_{i+1}$ by the action of $b_i$ on the base is isomorphic to $V_i$ as a $\U(1)$ principal bundle, completing the proof.
\end{proof}

\vspace{12pt}

\noindent\textbf{Flips.} There is an additional set of symmetries in $\mathcal{R}_{g,n}$ which we call ``flips''. For any subset $J\subset\{1,\ldots,n\}$ of even cardinality, we get an involution on $\wt{\mathcal{R}}_{g,n}$, denoted by $\wt M_J$:
$$\wt M_J(\ubar S,T_1,\ldots,T_n)=(\ubar S,\eps_1T_1,\ldots,\eps_nT_n)$$
where $\eps_k$ takes the value -1 if $k\in J$ and 1 if not. Since multiplication by $-1$ commutes with conjugation, this map descends to a map $M_J$ on $\mathcal{R}_{g,n}$. The maps $\wt M_J$ and $M_J$ are called \emph{flips}.

Suppose $\phi\in\Mod^*_{g,n}$ is a mapping class which permutes the punctures according to a permutation $\sigma$ of the indices. Then we have 
\begin{equation} \label{commuting_symmetries}
M_\phi\circ M_{\sigma(J)}=M_J\circ M_\phi
\end{equation}
Let us record the effect of the maps $M_J$ on the line bundles $V_k$ and submanifolds $D\pm_{k,l}$
\begin{lemma} \label{effect_flipping}
We have $M_J^*(V_k)\cong V_k^*$ if $k\in J$, otherwise $M_J^*(V_k)\cong V_k$. In addition, $M_J(D^\pm_{k,l})=D^\pm_{k,l}$ if $\abs{\{k,l\}\cap J}$ is even, and $M_J(D^\pm_{k,l})=D^\mp_{k,l}$ if it is odd.
\end{lemma}
\begin{proof}
The second statement is completely obvious, so we focus on the line bundles $V_k$. In the case that $k\notin J$, the map $\wt M_J$ simply maps $V_k$ to itself ($\U(1)$-equivariantly). So, we suppose $k\in J$. As an $S^1$ fiber bundle, $M_J^*(V_k)$ is just $\wt M_J^{-1}(V_k)$ by definition of pullback, which is just the set of representations $[\ubar S,\ubar T]$ with $T_k=-\mathbf{i}$. The map on $\wt{\mathcal{R}}_{g,n}$ which conjugates a representation by $\mathbf{j}$ brings $\wt M_J^{-1}(V_k)$ back to $V_k$, but this map is not $\U(1)$-equivariant. In fact, acting by $z\in\U(1)$ on $\wt M_J^{-1}(V_k)$ is the same as acting by $\obar z$ on $V_k$, so that these bundles are conjugates.
\end{proof} 

\vspace{12pt}

\noindent\textbf{Interaction with the Symplectic Structure.} The maps $M_J$ and $M_\phi$ have the convenient property that they are actually \emph{symplectomorphisms} of the sympletic manifold $(\mathcal{R}_{g,n},\omega)$.

\begin{prop} \label{prop_preserves_symp_form}
$M_J^*(\omega)=M_\phi^*(\omega)=\omega$.
\end{prop}
\begin{proof}
Let us first prove that $M_\phi^*(\omega)=\omega$. Let $[\rho]\in\mathcal{R}_{g,n}$, choose a connection $A\in\mathcal{A}_\text{flat}$ with $\hol_z(\su(A))=\rho$, and let $B=\wt M_\phi(A)$. The tangent space to the equivalence class $[\rho]\in\mathcal{R}_{g,n}$ consists of equivalence classes 1-forms with values in $\so(\Ad E)$ which vanish on the loci $s_k\geq 1/2$ near each puncture and are in the kernel of the connection operator $d_A$. For such a 1-form $a$, the pushforward 1-form $a'=(d\wt M_\phi)_A(a)$ is given at $x\in\Sigma^*$ by (letting $X\in T_x\Sigma^*$ and $v\in E_x$):
$$a'_x(X)\cdot v=\wh F_x\vphantom{F}^{-1}\left(a_{f(x)}(df(X))\cdot \wh F_x(v)\right)$$
Despite this complicated-looking formula, given two such 1-forms $a,b$, it is not hard to see that their pushforwards satisfy $\trace(a'\wedge b')=f^*\trace(a\wedge b)$, since the map $\wh F$ preserves the bilinear form $\trace$ on $\so(\Ad E)$. Hence:
\begin{align*}
M_\phi^*(\omega)\left([a]\wedge[b]\right)=&\frac{1}{4\pi^2}\int_\Sigma \trace\left(a'\wedge b'\right)\\
=&\frac{1}{4\pi^2}\int_\Sigma f^*\left(\trace\left(a\wedge b\right)\right)=\frac{1}{4\pi^2}\int_\Sigma \trace\left(a\wedge b\right)=\omega\left([a]\wedge[b]\right).
\end{align*}
This proves that elements of $\Mod_{g,n}$ preserve $\omega$.

As for the map $M_J$, we will need to lift this diffeomorphism to one on $\mathcal{A}_\text{flat}$ where the definition of $\omega$ originates; we need an operation on connections $A$ which negate the holonomies of $\su(A)$ around $x_k$ for $k\in J$. The idea is that if $A'$ is a connection with $\hol_z(\su(A'))=M_J(\hol_z(\su(A)))$, then $A$ and $A'$ will not be guage equivalent with respect to $\mathcal{G}$ but will be gauge equivalent with respect to the full gauge group $\wh{\mathcal{G}}$ of $\SO(3)$ automorphisms of $\Ad E$ over $\Sigma^*\setminus\{z_0\}$, not necessarily arising from an $\SU(2)$ automorphism of $E$. The group $\mathcal{G}$ is naturally a subgroup of $\wh{\mathcal{G}}$ of index $\abs{H^1(\Sigma;\ZZ/2)}$. Still using our trivialization of $E$ outside the auxillary point $z_0$, let $u_J$ be an $\SO(3)$ automorphism of $\Ad E$ over $\Sigma^*\setminus\{z_0\}$ such that with respect to this trivialization, $u_J$ gives a function $\Sigma^*\setminus\{z_0\}\to\SO(3)$ with $g(z)=\Id$ for whom the induced homomorphism ${u_J}_*:\wh\Gamma\to\pi_1(\SO(3),\Id)\cong\ZZ/2$ satisfies:
\begin{itemize}
\item ${u_J}_*(a_j)=\obar 0$, for $1\leq j\leq 2g$.
\item ${u_J}_*(d_k)=\obar 1$ for $k\in J$, and $\obar 0$ for $k\nin J$.
\item ${u_J}_*(\zeta)=\obar 0$, where $\zeta$ is the central order two generator.
\end{itemize}
Such a $u_J$ can be shown to exist using a homotopy equivalence of $\Sigma^*\setminus\{z_0\}$ with a wedge of circles. Moreover, we can arrange $u_J$ to be the identity near the puncture $x_k$ for $k\nin J$, and for $k\in J$, the explicit function:
\begin{equation}
u_J(s_k,\theta_k)=\left(\begin{smallmatrix}\cos\theta_k&0&-\sin\theta_k\\0&1&0\\ \sin\theta_k&0&\cos\theta_k\end{smallmatrix}\right)
=\Ad\left(\begin{smallmatrix}0&e^{i\theta_k/2}\\-e^{-i\theta_k/2}&0\end{smallmatrix}\right).
\end{equation}
Even though the matrix $\left(\begin{smallmatrix}0&e^{i\theta_k/2}\\-e^{-i\theta_k/2}&0\end{smallmatrix}\right)$ is not well defined on $U_k\setminus\{x_k\}$, the expression does give an adjoint which is single-valued. Near $x_k$ for $k\in J$, if $a$ is the fixed connection 1-form of $A$, the 1-form of $u_J\cdot A$ is
\begin{align*}
\Ad_{u_J}(a)+u_Jd(u_J^{-1})&\;=\ad\left[\tfrac{1}{4}\left(\begin{smallmatrix}0&e^{i\theta_k/2}\\-e^{-i\theta_k/2}&0\end{smallmatrix}\right) \left(\begin{smallmatrix}i&0\\0&-i\end{smallmatrix}\right)\left(\begin{smallmatrix}0&-e^{i\theta_k/2}\\e^{-i\theta_k/2}&0\end{smallmatrix}\right)\right. \\
&\hspace{36pt}\left.+\tfrac{1}{2}\left(\begin{smallmatrix}0&e^{i\theta_k/2}\\-e^{-i\theta_k/2}&0\end{smallmatrix}\right)\left(\begin{smallmatrix}0&ie^{i\theta_k/2}\\ie^{-i\theta_k/2}&0\end{smallmatrix} \right)\right]\\
&\;=\ad\left[\tfrac{1}{4}\left(\begin{smallmatrix}-i&0\\0&i\end{smallmatrix}\right)+\tfrac{1}{2}\left(\begin{smallmatrix}i&0\\0&-i\end{smallmatrix}\right)\right]=\tfrac{1}{4}\ad\left(\begin{smallmatrix}i&0\\0&-i\end{smallmatrix}\right)=a
\end{align*}
We see that acting by $u_J$ preserves $A$ near the punctures and so gives a map on $\mathcal{A}$, which certainly preserves the flat locus and descends to a function on $\mathcal{G}$ equivalence classes. The holonomy of $\su(u_J\cdot A)$ around the loop $d_k$ will clearly be that of $\su(A)$ for $k\nin J$. Moreover, since $u_J$ is just an $\SO(3)$ gauge transformation with $u_J(z)=\Id$, the $\SO(3)$ holonomy of $u_J\cdot A$ based at $z$ is exactly that of $A$. Suppose if $k\in J$ and let $s(\theta_k)$ denote a section $[0,2\pi]\to E$ over the loop $d_k$ that is $\su(A)$-parallel. Along $d_k:[0,2\pi]\to\Sigma^*$, with respect to the trivialization of $E$, $u_J$ is $\Ad v_J$ for some $v_J:[0,2\pi]\to\SU(2)$. Since $u_J$ does not extend across the $k$th puncture for $k\in J$, it must be the case that the lift $v_J$ takes the value $v_J(2\pi)=-1$ instead of $1$. Hence, the holonomy of $\su(u_J\cdot A)$ around $d_k$ must be $-1$ times that of $\su(A)$. This shows that $u_J$ lifts the map $M_J$ to $\mathcal{A}_\text{flat}$. Because $u_J$ is really just a gauge transformation, it certainly preserves the closed 2-form $\wt\omega$ on $\mathcal{A}_\text{flat}$. We conclude that $M_J$ preserves $\omega$ as desired.
\end{proof}

\vspace{12pt}

\noindent\textbf{Interaction with the Universal Bundle.} There is one last result on the effect of flips and mapping classes we will need when we begin working with cohomology classes.

\begin{prop} \label{prop_symmetries_univ_bundle}
For a mapping class $\phi\in\Mod_{g,n}$, let $f$ be the chosen representative, and let $\sigma$ be the corresponding permutation of the punctures. Then as bundles on $\mathcal{R}_{g,n}\times\Sigma$, we have
\begin{enumerate}[(i)]
\item $(M_\phi\times f^{-1})^*(\mathbf{E}^\ad)\cong\mathbf{E}^\ad.$
\item $M_\phi^*(\mathbf{V}_k)\cong\mathbf{V}_{\sigma(k)}.$
\end{enumerate}
For $J\subset\{1,\ldots,n\}$ with $\abs{J}$ even, we also have
\begin{enumerate}[(i)]
\setcounter{enumi}{2}
\item $(M_J\times\Id)^*(\mathbf{E}^\ad|_{\mathcal{R}_{g,n}\times\Sigma^*})\cong\mathbf{E}^\ad|_{\mathcal{R}_{g,n}\times\Sigma^*}$
\item $M_J^*(\mathbf{V}_k)\cong\begin{cases}\mathbf{V}_k\text{, if }k\nin J\\																					\mathbf{V}_k^*\text{, if }k\in J\end{cases}$
\end{enumerate}
\end{prop}
\begin{proof}
Recall that given the mapping class $\phi$ we have chosen a diffeomorphism representative $f$ with controlled behavior near the punctures, and lifting isomorphism $\widehat F:\Ad E\to\Ad E$. To prove (i), we will construct a fiberwise isomorphism of $\mathbf{E}^\ad$ to itself covering $M_\phi\times f^{-1}$. A point in the total space of $\mathbf{E}^\ad$ is a $\mathcal{G}$-equivalence class $\obar{(t,A,x)}$ with $t\in \Ad E_x$. We define the map via:
$$\left(\obar{(t,A,x)}\right)\mapsto\obar{\left(\wh F_x\vphantom{F}^{-1}(t),\wt M_\phi(A),f^{-1}(x)\right)}.$$
To show that this is independent of the choice of representative $(t,A,x)$, let $g\in\mathcal{G}$. A different representative is $g\cdot(t,A,x)=(g_x\circ t\circ g_x^{-1}, g\cdot A,x)$, which gets sent to
\begin{equation} \label{F_phi_diff_rep}
\left(\wh F_x\vphantom{F}^{-1}(g_x\circ t\circ g_x^{-1}),M_\phi(g\cdot A),f^{-1}(x)\right)
\end{equation}
Since $M_\phi(g\cdot A)=(F^{-1}\circ g\circ F)\cdot M_\phi(A)$, it is clear that acting by the gauge transformation $$F^{-1}\circ g\circ F \text{ on }\left(\wh F_x\vphantom{F}^{-1}(t),\wt M_\phi(A),f^{-1}(x)\right)$$
gives exactly the output (\ref{F_phi_diff_rep}), proving independence of gauge representative. Hence, the recipe (\ref{F_phi_diff_rep}) is well defined on equivalence classes and gives the desired isomorphism of bundles. The second isomorphism (ii) of line bundles follows by an entirely analogous argument.

For (iii), we recall from the proof of Proposition \ref{prop_preserves_symp_form} that there is an $\SO(3)$ gauge transformation $u_J$ lifting $M_J$ to the space of flat connections. We define a fiberwise isomorphism from $\mathbf{E}^\ad$ to itself covering $M_J\times\Id$ via:
\begin{equation} \label{flip_on_univ_map}
\obar{(t,A,x)}\mapsto\obar{\left( (u_J)_x\cdot t,u_J\cdot A,x\right)}.
\end{equation}
It is exactly because $u_J$ does not extend across the puncture that we cannot extend this map to one on $\mathbf{E}^\ad$ over all of $\Sigma$. To see independence on gauge representative, for $g\in\mathcal{G}$ the different representative $(g_x\circ t\circ g_x^{-1},g\cdot A,x)$ is sent to:
$$\left((u_J)_x(g_x\circ t\circ g_x^{-1}),u_J\cdot(g\cdot A),x\right)=(u_J\circ\Ad_g\circ u_J^{-1})\cdot \left((u_J)_x(t),u_J\cdot A,x\right).$$
It remains to show that $u_J\circ\Ad_g\circ u_J^{-1}=\Ad_{g'}$ for a gauge tranformation $g'\in\mathcal{G}$. A determinant 1 automorphism $g'$ of $E$ satisfying this equation will certainly exist and be unique up to sign but it remains to check that $g'$ is a constant diagonal element of $\SU(2)$ near a puncture. This is obvious when $k\nin J$. Near the $k$th puncture for $k\in J$, we have:
\begin{align*}
u_J\circ\Ad_g\circ u_J^{-1}&\;=\Ad\left(\begin{smallmatrix}0&e^{i\theta_k/2}\\-e^{-i\theta_k/2}&0\end{smallmatrix}\right)\circ \Ad\left(\begin{smallmatrix}w&0\\0&\obar w\end{smallmatrix}\right)\circ\Ad\left(\begin{smallmatrix}0&-e^{i\theta_k/2}\\e^{-i\theta_k/2}&0\end{smallmatrix}\right) \\
&\;=\Ad\left(\begin{smallmatrix}\obar w&0\\0&w\end{smallmatrix}\right)
\end{align*}

We prove the final isomorphism (iv) in the case $k\in J$. Let $h\in H_k\subset\Ad E_x$ so that $h=\ad\left(\begin{smallmatrix}0&-\obar v\\v&0\end{smallmatrix}\right)$ for $v\in\CC$. We define a fiberwise bijection $\mathbf{V}_k\to\mathbf{V}_k$ covering $M_J$ on $\mathcal{R}_{g,n}\times\{x_k\}$ via:
\begin{equation} \label{V_k_M_J_map}
\obar{\left(\ad\left(\begin{smallmatrix}0&-\obar v\\v&0\end{smallmatrix}\right),A\right)}\mapsto\obar{\left(\ad\left(\begin{smallmatrix}0&-v\\ \obar v&0\end{smallmatrix}\right),u_J\cdot A\right)}.
\end{equation}
Suppose $g_{x_k}=\diag(w,\obar w)$. The different representative
$$g\cdot (h,A)=\left(\ad\left[\left(\begin{smallmatrix}w&0\\0&\obar w\end{smallmatrix}\right)\left(\begin{smallmatrix}0&-\obar v\\v&0\end{smallmatrix}\right)\left(\begin{smallmatrix}\obar w&0\\0&w\end{smallmatrix}\right)\right],g\cdot A\right)=\left(\ad\left(\begin{smallmatrix}0&-w^2\obar v\\ \obar w^2v&0\end{smallmatrix}\right),g\cdot A\right)$$
is sent to $\left(\ad\left(\begin{smallmatrix}0&-\obar w^2v\\w^2\obar v&0\end{smallmatrix}\right),u_J\cdot(g\cdot A)\right)$. It is not hard to check that this is the same as acting by $u_J\circ\Ad_g\circ u_J^{-1}$ on $\left(\ad\left(\begin{smallmatrix}0&-v\\ \obar v&0\end{smallmatrix}\right),u_J\cdot A\right)$. The proof is completed by noting that  (\ref{V_k_M_J_map}) is complex conjugate-linear, and so gives a fiberwise isomorphism $\mathbf{V}_k\to\mathbf{V}_k^*$ covering $M_J$.
\end{proof}

\begin{prop} \label{prop_values_r_s}
In Corollary \ref{V_k_poincare_dual} above, we have $\abs{r}=\abs{s}=1$.
\end{prop}
\begin{proof}
Let us first treat the case $n\geq5$. By Lemma \ref{mod_gn_action_on_V_k}, we can assume without loss of generality that $k=3$ and $l=4$. We will use an auxillary submanifold to cut down the moduli space so that the necessary computation may be performed on a copy of the two-sphere. Let
$$\mathcal{S}\subset\mathcal{R}_{g,n}$$
denote the subset of points $[\ubar S,\ubar T]$ where the $S_j$'s are all 1. There is an obvious identification $\mathcal{S}\leftrightarrow\mathcal{R}_{0,n}$ which carries $D^\pm_{k,l}\cap \mathcal{S}\subset\mathcal{R}_{g,n}$ to $D^\pm_{k,l}\subset\mathcal{R}_{0,n}$ and for which the pullback of $V_k$ is just $V_k$. Hence, we can assume without a loss of generality that $g=0$. Define another submanifold $Z\subset\mathcal{R}_{0,n}$ by:
$$Z=D^+_{4,5}\cap D^+_{5,6}\cap \cdots \cap D^+_{n-1,n}$$
As a general rule, all intersections of the above type will turn out to be transverse. An outline of a proof of transversality in this case would go as follows. Observe first that issues of transversality can be dealt with upstairs in $\wt{\mathcal{R}}_{0,n}$ with the preimages of the $D^+_{k,l}$'s. Next, argue that show that any 2-way intersection is transverse by reducing to the case of $n=5$. Then, note that everytime another pair is added to the intersection, we are really studying an intersection in representation variety for 2 fewer points. The general case then follows by induction.

A point in $Z$ is an equivalence class $[T_1,T_2,T_3,T_4,\ldots,T_4]$, where $T_4$ appears $n-3=2m-2$ (and so an even number of) times. Hence, we must have $T_1T_2T_3=(-1)^m$. Up to conjugation, the representation is $[(-1)^{m+1}\mathbf{j},\mathbf{k},\mathbf{i},T_4,\ldots,T_4]$, which defines $T_4$ uniquely. This gives both a map $f:Z\mapsto C_\mathbf{i}$, as well as a trivializing section of $V_3|_Z$. To compute $r$ and $s$, we will study this restriction $V_3|_Z$ near the intersections $D^\pm_{3,4}\cap Z$. These intersections are transverse, and each is easily seen to be isomorphic to a representation variety for three parabolic points: a single point. Denote these intersection points by $\rho^\pm$. 

In order to compute the values of $r$ and $s$, we need to compute the winding number of the trivialization $\tau_{3,4}$ of $V_3|_{Z\setminus \{\rho^\pm\}}$ coming from Lemma \ref{V_k_trivialization} and $D^+_{3,4}$, with respect to the full trivialization $\tau'$ over all of $Z$. We will study the trivializations along the loop $\gamma:S^1\to C_\mathbf{i}$ defined by
$$\gamma(e^{i\theta})=[(-1)^{m+1}\mathbf{j},\mathbf{k},\mathbf{i},e^{i\theta}\mathbf{j},\ldots,e^{i\theta}\mathbf{j}]$$
where the multiplication occurs in the unit quaternions. For a point $\gamma(e^{i\theta})$ in the loop, the trivialization $\tau_{3,4}$ requires us to find a number in the complex circle $S^1_\mathbf{i}$ by which we may conjugate $e^{\mathbf{i}\theta}\cdot\mathbf{j}$ so that it just becomes $\mathbf{j}$. This number is just $e^{i\theta/2}$ (or its opposite). Since the trivialization $\tau'$ is just given by
$$\tau'\left([(-1)^{m+1}\mathbf{j},\mathbf{k},\mathbf{i},e^{i\theta}\mathbf{j}]\right)=\left((-1)^{m+1}\mathbf{j},\mathbf{k},\mathbf{i},e^{i\theta}\mathbf{j}\right),$$
we see that the difference between the two trivialization is conjugating by $e^{\mathbf{i}\theta/2}$. Recall that the $\U(1)$ structure of $V_3$ is by the quotient of $S^1_\mathbf{i}$ by $\{\pm 1\}$, so the ``clutching'' function is really just $e^{i\theta}\mapsto e^{i\theta}$ and so is of degree $\pm 1$. There is an ambiguity in signs arising from the lack of orientation chosen for the $D^+_{l,k}$'s and general disregard for sign conventions regarding the definition and computation of the first Chern class. This completes the proof in the case $n\geq 5$.

If $n=3$, consider the embedding $\iota:\mathcal{R}_{g,3}\into\mathcal{R}_{g,5}$ given by
$$[\ubar S,T_1,T_2,T_3]\mapsto[\ubar S,T_1,T_2,T_3,-T_3,T_3]$$
We may assume without loss of generality that $k,l=1,2$. It is straightforward to see that $V_1(g,3)\cong\iota^*V_1(g,5)$. Since $D^\pm_{k,l}(g,3)=\iota^{-1}(D^\pm_{k,l}(g,5))$, the result follows from the $n=5$ case.
\end{proof}

\section{Relations in the Cohomology Ring} \label{section_relations}

We know from Theorem \ref{thm_generating_set} that the cohomology ring of $\mathcal{R}_{g,n}$ comes with a generating set of classes of degrees 2, 3, and 4, arising from a universal bundle pair $(\mathbf{E}^\ad,\{\mathbf{V}_k\})$ on $\mathcal{R}_{g,n}\times\Sigma$. In order to understand the relations between these classes, it is necessary to describe them more concretely. For example, by the definition of slant product $/$ the degree four class $p_1(\mathbf{E}^\ad)/[\point]$ is just $p_1(\mathbf{E}^\ad|_{\mathcal{R}_{g,n}\times\{\point\}})$. This, we know by Lemma \ref{equiv_bundles_over_R}, is just $p_1(W)$, where again W is the rank three $\RR$ vector bundle associated to the $\SO(3)$ bundle $\wt{\mathcal{R}}_{g,n}\to\mathcal{R}_{g,n}$.

In \S\ref{subsection_univ_bundles_and_V_ks} we identified the classes $c_1(\mathbf{V}_k)\in H^2(\mathcal{R}_{g,n})$ with the classes $c_1(V_k)$, and we found explicit submanifolds representating their Poincar\'e duals. These submanifolds have the property that intersections between two of them (with one common index) behave like representation variety for two fewer marked points. The key to exploiting this property is identifying the restrictions of the other cohomology classes to these intersections with the corresponding classes in the lower representation varieties. This will give information about how the pairings of these other classes with the $c_1(V_k)$'s behave. In what follows, for a homology class $h\in H_*(\Sigma)$, we will use the simple shorthand $[h]$ to denote $p_1(\mathbf{E}^\ad)/[h]$, so that, for example, $[\text{pt}]=p_1(W)$.

\subsection{The Class of the Symplectic Structure}

There is another natural cohomology class on $\mathcal{R}_{g,n}$: the degree two class $[\omega]$ of the symplectic form. By Corollary \ref{cor_rank_degree_2} and Theorem \ref{thm_generating_set}, we know that $H^2(\mathcal{R}_{g,n};\QQ)$ is spanned by the classes $[\Sigma]$ and $c_1(V_1),\ldots,c_1(V_n)$. Hence, $[\omega]$ must be a linear combination of these:
\begin{equation} \label{lin_comb_symp_class}
[\omega]=A[\Sigma]+\sum_{k=1}^nD_kc_1(V_k)
\end{equation}
By symmetry the $D_k$'s must all be equal. We claim that $A\neq 0$. To see this, suppose on the contrary that $[\omega]=D\sum_{k=1}^nc_1(V_k)$. We average $M_J^*$ applied to this equation over all even $J$. The number of even $J$ containing an index is the same as the number not containing it and $M_J^*(c_1(V_k))=\pm c_1(V_k)$ depending on whether $k\in J$, so the right hand side of the averaged equation vanishes. However $M_J^*([\omega])=[\omega]$, so we arrive at $[\omega]=0$. But $[\omega]$ cannot be zero as it is the class of a symplectic form on a compact manifold. We conclude that $A\neq 0$, and we have proved:

\begin{prop} \label{symp_class_basis}
The $n+1$ classes $[\omega],c_1(V_1),\ldots,c_1(V_n)$ are a basis for $H^2(\mathcal{R}_{g,n})$.
\end{prop}

The class $[\omega]$ is more convenient than $[\Sigma]$ for us because it behaves well under the symmetries, including flips. What we lack is a geometric description for the Poincar\'e dual of $[\omega]$. Instead, for us the important data concerning this class will be the pairing of its top power with the fundamental class of the moduli space, called the symplectic volume. The paper \cite{weitsman_toric_I} gives a formula for this top pairing in the case of arbitrary rational weights $t_k$. We note that in that paper, their weights in $(0,1)$ correspond to twice the value of our weights in $(0,1/2)$. We have rewritten their formula to agree with our conventions.

\begin{thm} \label{thm_weitsman_symp_volume}
(\cite{weitsman_toric_I}, Prop. 4.12, and eq. (5.3)) The pairing of the top power of $[\omega]$ with the moduli space is given by:
\begin{equation} \label{symp_volume_limit}
\langle[\omega]^{3g-3+n},\mathcal{R}_{g,n}(\ubar t)\rangle=\frac{(3g+n-3)!}{2^{g-2}\pi^{2g-2+n}}\sum_{N=1}^\oo\frac{1}{N^{2g-2+n}}\prod_{k=1}^n\sin(2\pi N t_k)
\end{equation}
\end{thm}

\begin{cor} \label{cor_symp_volume}
For the case $n\geq 1$ and odd and $\ubar t=(1/4,\ldots,1/4)$, we have
\begin{equation} \label{symp_volume}
\langle[\omega]^{3g+n-3},\mathcal{R}_{g,n}(\ubar t)\rangle=\frac{(3g+n-3)!}{2^{3g+n-3}g!}\abs{E_{2g+n-3}}
\end{equation}
where $E_j$ is the $j$th Euler number defined to be the coefficient of $x^j/j!$ in the Taylor series of $\sech(x)=1/\cosh(x)$.
\end{cor}
\begin{proof}
We plug in the weight $t_k=1/4$ in (\ref{symp_volume_limit}) and simplify. For these values of $t_k$, the even $N$ terms in the sum in (\ref{symp_volume_limit}) vanish, are we are left with:
\begin{equation} \label{symp_volume_inf_sum}
\langle[\omega]^{3g-3+n},\mathcal{R}_{g,n}(\ubar t)\rangle=\frac{(3g+n-3)!}{2^{g-2}\pi^{2g-2+n}}\sum_{M=1}^\oo\frac{(-1)^M}{(2M+1)^{2g-2+n}}
\end{equation}
The sum now appearing consists of alternating negative powers of the odd integers, which is a well documented function of the exponent $2g-2+n$ known as a the Dirichlet $\beta$ function. It can be computed exactly, and one has:
$$\beta(2l+1)=\frac{(-1)^l\pi^{2l+1}E_{2l}}{2^{2l+2}(2l)!}$$
Plugging $\beta(2g-2+n)$ in for the sum in (\ref{symp_volume_inf_sum}) gives the desired formula.
\end{proof}

\subsection{Inductive Properties of the Moduli Space}

In order to get information about pairings with the classes $c_1(V_k)$, let us further analyze the $D^\pm_{k,l}$'s and their intersections. Given a pair $(k,l)$ of indices, we can construct a map  to a smaller representation variety $D^\pm_{k,l}\to\mathcal{R}_{g,n-2}$ in the following way. Let $b\in B_n$ be such that the correpsonding permutation carries $(n-1,n)$ to $(k,l)$. Then $b$ gives a diffeomorphism from $D^\pm_{k,l}$ to $D^\pm_{n-1,n}$. There is then a map $D^\pm_{n-1,n}\to\mathcal{R}_{g,n-2}$ given by:
\begin{equation} \label{recursive_map}
[\ubar S,T_1,\ldots,T_{n-2},T_{n-1},\pm T_{n-1}]\mapsto[\ubar S,T_1,\ldots,T_{n-3},\mp T_{n-2}]
\end{equation}
Indeed, $T_n\cdot T_n=-1$. This map is clearly surjective, and its fiber is just the freedom is choosing $T_n$: it is a copy of the sphere $C_{\mathbf i}$. Hence, by composing it with $b$, we see that $D^\pm_{k,l}$ is a $C_{\mathbf i}$ fiber bundle over $\mathcal{R}_{g,n-2}$.

\begin{remark}
This situation is to be contrasted with that in \cite{weitsman_cohom}, where all of the parabolic weights $t_k$ are distinct. There, the analogue of $D^\pm_{k,l}$ is also a (connected component of a) subspace where $T_k$ and $T_l$ commute but admits an isomorphism to an actual representation variety for one fewer parabolic point, where the parabolic weights $t_k$ and $t_l$ are replaced by the single weight $t_k\pm t_l$. Such a representation variety for us does not exist (smoothly), of course.
\end{remark}

Consider now an intersection $D^+_{j,k}\cap D^+_{k,l}$ with $j,k,l$ distinct. Rechoose $b$ so that its permutation carries $(n-2,n-1,n)$ to $(j,k,l)$, so $b$ gives a diffeomorphism $D^+_{j,k}\cap D^+_{k,l}$ to $D^+_{n-2,n-1}\cap D^+_{n-1,n}$. There is then a map from $D^+_{n-2,n-1}\cap D^+_{n-1,n}$ to $\mathcal{R}_{g,n-2}$ given by restricting the map (\ref{recursive_map}):
\begin{equation} \label{recursive_map_II}
[\ubar S,T_1,\ldots,T_{n-2},T_{n-2},T_{n-2}]\mapsto[\ubar S,T_1,\ldots,T_{n-3},-T_{n-2}]
\end{equation}
which is a diffeomorphism. Composed with the map induced by $b$, we get a diffeomorphism of $D^+_{j,k}\cap D^+_{k,l}$ with $\mathcal{R}_{g,n-2}$. There are many such maps as a result of, for example, the choice of the braid $b$, but they all differ by postcomposing with the maps from flips and mapping classes on $\mathcal{R}_{g,n-2}$. There are similar isomorphisms for intersections $D^\pm_{j,k}\cap D^\pm_{k,l}$. Moreover, the inverse of (\ref{recursive_map_II}) gives a section of the 2-sphere fiber bundle (\ref{recursive_map}). It is also easy to see that \ref{recursive_map_II} carries $D^\pm_{a,b}\cap D^+_{n-2,n-1}\cap D^+_{n-1,n}$ to the corresponding $D^\pm_{a,b}$ inside $\wt{\mathcal{R}}_{g,n-2}$, for $a,b\leq n-3$, and to $D^\mp_{a,b}$ if one of $a,b$ is $n-2$, $n-1$, or $n$. The moral is that successive intersections of the $D^\pm_{k,l}$'s behave like representation varieties for fewer parabolic points. This recursive property is the key to understanding the cohomology of $\mathcal{R}_{g,n}$.

\begin{prop} \label{univ_bundle_recursive}
Let $\iota:D^\pm_{k,l}\to\mathcal{R}_{g,n}$ be the inclusion, and $\pi:D^\pm_{k,l}\to\mathcal{R}_{g,n-2}$ be a map arising from from formula (\ref{recursive_map}) and the discussion preceding it, and let $\tau:\{1,\ldots,n-2\}\into\{1,\ldots,n\}$ denote the corresponding inclusion of index sets. If $\mathbf{E}^\ad(g,n)$ denotes the universal bundle on $\mathcal{R}_{g,n}\times\Sigma$, then we have:
\begin{itemize}
\item $\mathbf{E}^\ad(g,n)|_{D^\pm_{k,l}\times\Sigma}\cong(\pi\times\id)^*\mathbf{E}^\ad(g,n-2)$
\item $\mathbf{V}_{\tau(k)}(g,n)|_{D^\pm_{k,l}\times\Sigma}\cong(\pi\times\id)^*\mathbf{V}_k(g,n-2)\text{ or }(\pi\times\id)^*\mathbf{V}_k(g,n-2)^*$
\end{itemize}
\end{prop}
\begin{proof}
As in the proof of Proposition \ref{prop_symp_form_recursive}, we again assume we are working with $D^-_{n-1,n}$, let $\pi$ be the map (\ref{n1n_proj_map}) and we use all the same notation from there. Let $\wh{\mathcal{A}}_\nu$ denote the space of flat connections on $\wh\Sigma^*$ arising from the operation of extending connections in $\mathcal{V}$ by the product connection; it is the space of flat connections on a surface with two fewer punctures whose 1-forms with respect to the induced trivialization on $\wh\Sigma^*\setminus\{z_0\}$ are zero on a disk $U$. Denote the composition $\mathcal{A}_\nu\to\mathcal{V}\to\wh{\mathcal{A}}_\nu$ by using a hat. Now, since $\mathcal{A}_\nu/\mathcal{G}_\nu=\mathcal{A}_\text{flat}/\mathcal{G}$, the bundles $\mathbf{E}^\ad(g,n)$ and $\mathbf{V}_k(g,n)$ can also be constructed as the quotient of the pullback of $\Ad E$ and $H_k$ to $\mathcal{A}_\nu\times\Sigma$ and $\mathcal{A}_\nu\times\{x_k\}$. We claim that a fiberwise isometry $\mathbf{E}^\ad(g,n)\to\mathbf{E}^\ad(g,n-2)$ covering $\pi\times\id$ may be defined by the formula:
$$\obar{(t,A,x)}\mapsto\obar{(t,\wh A,x)}$$
for $A\in\mathcal{A}_\nu$ and $t\in\Ad E_x$. Note that the underlying compact surface $\Sigma$ and bundle $E$ for any number of punctures is the same, so we are free to repeat ``$x$'' and ``$t$'' on the right hand side. It is obvious that this is well-defined: if $g\in\mathcal{G}_\nu$ and $\wh g$ is the tranformation $g$ replaced by 1 on $U$, then $g\cdot(t,A,x)$ is sent to $\wh g\cdot(t,\wh A,x)$. This proves the first bundle isomorphism. The second is proved via an almost identical argument. The ambiguity between the line bundle and its dual arises from the effect of flips on the $\mathbf{V}_k$'s.
\end{proof}

As a corollary, we see that the classes $[\point]$ and $[a_j]$ restricted to $D^\pm_{k,l}$ are all pulled back from $\mathcal{R}_{g,n-2}$ in the expected way. Finally, we need to study how the projection $D^\pm_{k,l}\to\mathcal{R}_{g,n-2}$ behaves with respect to the class $[\omega]$. We have:

\begin{prop} \label{prop_symp_form_recursive}
Denote by $\omega_{g,n}$ the symplectic form on $\mathcal{R}_{g,n}$. Let $\pi:D^{\pm}_{k,l}\to\mathcal{R}_{g,n-2}$ be any of the natural 2-sphere fiber bundles arising from the map (\ref{recursive_map}), and let $\iota:D^{\pm}_{k,l}\hookrightarrow\mathcal{R}_{g,n}$ be the inclusion. Then we have $\iota^*(\omega_{g,n})=\pi^*(\omega_{g,n-2})$.
\end{prop}
\begin{proof}
We illustrate the Proposition in the case $D^-_{k,l}$, as the other case follows by invoking Lemma \ref{effect_flipping}. Without loss of generality, we may take $k,l$ = $n-1,n$, and $\pi$ is the map:
\begin{equation} \label{n1n_proj_map}
[\ubar S,T_1,\ldots,T_{n-2},T_{n-1},-T_{n-1}]\mapsto [\ubar S,T_1,\ldots,T_{n-2}].
\end{equation}
Any of the other natural choices for $\pi$ differ by postcomposing with the known symplectomorphisms of $\mathcal{R}_{g,n-2}$. We need to realize this map as an operation on flat connections in $\mathcal{A}$ in order to understand its interaction with the symplectic form. It will be convenient to work with $\SU(2)$-connections, so fix a basepoint $z\in\Sigma^*$ away from the punctures and recall our trivialization of $E$ away from the auxillary point $z_0$. Connections are now the same as $\su(2)$-valued 1-forms on $\Sigma^{**}=\Sigma^*\setminus\{z_0\}$, and these 1-forms are fixed near the punctures. Let $U\in\Sigma$ be a disk containing the neighborhoods of only the punctures $x_{n-1},x_n$ and such that $z\in\partial U$ and $z_0\nin U$, let $\widehat\Sigma^*$ denote the surface obtained by filling in these two punctures, and set $V=\Sigma\setminus U$. Let $\mathcal{A}^-_{n-1,n}$ denote the space of connections $A$ for whom $[A]\in D^-_{n-1,n}$. Every connection in $\mathcal{A}^-_{n-1,n}$ has $\SU(2)$ holonomy around $\partial U$ equal to the product of two antipodal elements of $\SU(2)$, which is just the identity. Since $\pi_1(\SU(2))=1$, for any $A\in\mathcal{A}^-_{n-1,n}$ we can find a gauge equivalent $A'$ such that the 1-form of $\su(A')$ vanishes on a fixed annular neighborhood $\nu$ of $\partial U$. Let $\mathcal{A}_\nu$ denote the subset of such (flat) connections; this subset is acted on by the subgroup $\mathcal{G}_\nu$ of gauge transformations which are constant on $\nu$ and the quotient is all of $D^-_{n-1,n}$.

Let $\mathcal{U},\mathcal{V}$ denote the spaces of flat connections on $U,V$ respectively, with the desired behavior near the punctures and whose 1-forms vanish near the boundary. Restriction of connections gives a homeomorphism $\wt\eta:\mathcal{A}_\nu\to\mathcal{U}\times\mathcal{V}$. If $\wt\omega_\mathcal{U}$ and $\wt\omega_\mathcal{V}$ denote the 2-forms on $\mathcal{U}$ and $\mathcal{V}$ coming from restricting the domain of integration in the definition (\ref{def_symp_form}) of $\wt\omega$ on $\mathcal{A}_\nu\subset\mathcal{A}$, then splitting the domain of integration implies that
\begin{equation} \label{symp_form_sum}
\wt\eta^*(\wt\omega_\mathcal{U}\oplus\wt\omega_\mathcal{V})=\wt\omega
\end{equation}
Letting $\mathcal{G}^1_\nu\subset\mathcal{G}_\nu$ denote the subgroup of $g$ with $g|_\nu=1$, the map $\wt\eta$, the 2-forms in equation (\ref{symp_form_sum}), and the relationship \ref{symp_form_sum} descend to
$$\wt\eta':\mathcal{A}_\nu/\mathcal{G}^1_\nu\to\mathcal{U}/\mathcal{G}^1_\nu\times\mathcal{V}/\mathcal{G}^1_\nu,$$
2-forms $\wt\omega'_\mathcal{V}$, $\wt\omega'_\mathcal{U}$, and $\wt\omega'$ and the equation $\wt\eta'^*(\wt\omega'_\mathcal{U}\oplus\wt\omega'_\mathcal{V})=\wt\omega'$. Now, given $A'\in\mathcal{V}$ and $A''\in\mathcal{U}$, there are corresponding connections $\wh A'$ and $\wh A''$ on $\wh\Sigma^*$ and $\wh U$, the sphere gotten by capping off $U$ with another disk, obtained by extending via the trivial connection. This identifies (using the ideas of Lemma \ref{partial_quotient}) the factor $\mathcal{U}/\mathcal{G}^1_\nu$ with the space
$$\wt{\mathcal{R}}_U=\{T_{n-1},T_n\in C_\mathbf{i}:\;T_{n-1}T_n=1\}\cong C_\mathbf{i}$$
and the factor $\mathcal{V}/\mathcal{G}^1_\nu$ with $\wt{\mathcal{R}}_{g,n-2}$. The space $\mathcal{A}_\nu/\mathcal{G}^1_\nu$ is clearly just the preimage $\wt D^-_{n-1,n}$ in $\wt{\mathcal{R}}_{g,n}$. Implicitly using all these identifications, it is straightforward to check that $\wt\omega'_\mathcal{V}$ corresponds to the closed 2-form $\wt\omega'_{g,n-2}$ on $\wt{\mathcal{R}}_{g,n-2}$ which is the pullback of $\omega_{g,n-2}$, and $\wt\omega'$ is the pullback of $\omega_{g,n}$ on $\mathcal{R}_{g,n}$. We have achieved now the isomorphism of manifolds with 2-forms:
$$\wt\eta':(\wt D^-_{n-1,n},\wt\omega'_{g,n})\cong (C_\mathbf{i},\wt\omega'_\mathcal{U})\times(\wt{\mathcal{R}}_{g,n-2},\wt\omega'_{g,n-2})$$
The quotient $\mathcal{G}_\nu/\mathcal{G}^1_\nu$ is isomorphic to $\SU(2)$ and the residual action of this group on $\wt{\mathcal{R}}_{g,n-2}$ is projectively free. The quotient on both sides (where the action on the right hand side is the simultaneous one) gives the $C_\mathbf{i}$-fiber bundle $\pi$. We are done then if we can show that $\wt\omega'_\mathcal{U}$ is actually 0. The key point is that all the connections in $\mathcal{U}$ are actually gauge equivalent; the quotient of $C_\mathbf{i}$ by $\SU(2)$ is a single point. Any two tangent vectors in $T_{A''}\mathcal{U}$ then differ by $(dg)g^{-1}$ for some gauge transformation $g$, which is in the annihilator of the linear form $\wt\omega'_\mathcal{U}$ (see the discussion on the symplectic structure in \S\ref{subsec_mod_space}). Hence, $\wt\omega'_\mathcal{U}$ must vanish.
\end{proof}

\begin{cor} \label{cor_symp_form_recursive}
The restriction of the map $\pi$ in \ref{prop_symp_form_recursive} to the intersection $D^{\eps_1}_{j,k}\cap D^{\eps_2}_{k,l}$ is a symplectomorphism.
\end{cor}
\begin{proof}
This restriction is a diffeomorphism by the discussion preceding Proposition \ref{prop_values_r_s}, and its inverse gives a section of the 2-sphere bundle $\pi$.
\end{proof}

Up until now, we have been entirely focused on the line bundles $V_k$ and degree two classes $c_1(V_k)$. We now consider the degree three classes $[a_j]$ for $1\leq j\leq 2g$. Their role in the cohomology ring has been well understood for over twenty years, as we now review. As shown in \cite{thaddeus_conformal}, it is convenient to introduce the class $\gamma_j=\tfrac{1}{16}[a_{2j-1}][a_{2j}]$ (our normalization here will be justified later on) for $1\leq j\leq g$. For each $j$, there is a natural embedding $\iota_j:\mathcal{R}_{g-1,n}\into\mathcal{R}_{g,n}$ given by
$$[\ubar S,\ubar T]\mapsto[S_1,\ldots,S_{2j-2},1,1,S_{2j-1},\ldots,S_{2g-2},\ubar T]$$
whose image is exactly the collection of representations $[\ubar S,\ubar T]$ for which $S_{2j-1}=S_{2j}=1$. In \cite{thaddeus_conformal}, it is proved that the submanifold $\iota_j(\mathcal{R}_{g-1,n})$ is Poincar\'e dual to $\gamma_j$ (at least in the case $n=0$, and the proof adapts readily to our situation). It is straightforward to check using the methods of the current paper that $\iota_j$ respects the symplectic forms, and the universal bundle pair over $\mathcal{R}_{g-1,n}$ is pulled back via $\iota_j$. This fact shows that the cohomology ring also has an inductive structure in the genus $g$. Putting everything together, we have:

\begin{prop} \label{prop_cut_down_relation}
The product class $\gamma_{j_1}\cdots\gamma_{j_r}c_1(V_{k_1})\cdots c_1(V_{k_s})$ is a constant multiple of the Poincar\'e dual to a collection of $2^s$ submanifolds of $\mathcal{R}_{g,n}$ each symplectomorphic to a copy of $\mathcal{R}_{g-r,n-2s}$ by a map under which the classes $[\point]$, $[a_j]$, and $c_1(V_k)$ (where $j\neq 2j_i,2j_i-1$ and $k\neq k_i$ for any $i$) are all pulled back accordingly.
\end{prop}
\begin{proof}
This is all a straightforward synthesis of facts proved up to this point. The result follows by induction after proving it in the case $r,s=0,1$ or $1,0$. The case when the class is $c_1(V_k)$ for some $k$ follows from Propositions \ref{univ_bundle_recursive} and \ref{prop_symp_form_recursive}. The case of the class $\gamma_j$ follows from \cite{thaddeus_conformal} and the discussion preceding the proposition.
\end{proof}

\begin{cor} \label{cor_smaller_relations}
Suppose $f(a,b)$ is a polynomial such that $f([\omega],[\point])\in H^*(\mathcal{R}_{g,n};\QQ)$ equals the zero class. Then the polynomial
$$\gamma_{j_1}\cdots\gamma_{j_r}c_1(V_{k_1})\cdots c_1(V_{k_s})\cdot f([\omega],[\point])=0$$
equals the zero class in the ring $H^*(\mathcal{R}_{g+r,n+2s};\QQ)$.
\end{cor}

\subsection{The Four Dimensional Class of a Point}

We would like to record some properties of the degree four class $[\point]$ and the degree three classes $[a_j]$. We first prove an easy equation relating $[\point]$ to the degree two classes $c_1(V_k)$.

\begin{lemma} \label{lem_point_class}
For any $k$, we have $[\point]=-c_1(V_k)^2$
\end{lemma}
\begin{proof}
We will prove this by studying a corresponding isomorphism of universal bundles on $\mathcal{R}_{g,n}$. By definition $[\point]=p_1(\mathbf{E}^\ad|_{\mathcal{R}_{g,n}\times\{\point\}})$ where $\point\in\Sigma$. Since the isomorphism type of $\mathbf{E}^\ad|_{\mathcal{R}_{g,n}\times\{\point\}}$ is independent of the choice of $\point$ as $\mathcal{R}_{g,n}$ is connected, we are free to choose $\point=x_k$ for any $k$. Write $\mathbf{E}^\ad_k\to\mathcal{R}_{g,n}$ for the restriction of $\mathbf{E}^\ad$ in this case. By definition of $\mathbf{V}_k$, it is clear that $\mathbf{E}^\ad_k$ is isomorphic as a real vector bundle to $\ubar \RR\oplus\mathbf{V}_k$. Hence, $p_1(\mathbf{E}^\ad_k)=-c_1(\mathbf{V}_k)^2$.
\end{proof}
\begin{remark}
One can prove Lemma \ref{lem_point_class} directly by studying the associated vector bundle constructions of $W$ and $V_k$ from $\PU(2)$ and $\U(1)$ bundles. Let $L_k$ again denote the line bundle $V_k\times_{\U(1)}\CC$ and $W$ the associated bundle $\wt{\mathcal{R}}_{g,n}\times_{\SO(3)}\RR^3$. Then for $\obar{(\rho,w)}$ in $V_k$, a map $L_k\oplus\ubar\RR\to W$ is given by
$$\left(\obar{(\rho,w)},r\right)\mapsto \left(\rho,\left(\begin{matrix}r&w \\ -\obar w&-r\end{matrix}\right)\right)$$
\end{remark}


\begin{lemma} \label{lem_point_class_pd}
The class $[\point]$ is Poincar\'e dual to a union of four, disjoint, codimension four submanifolds $D_1$, $D_2$, $D_3$, and $D_4$ each with a symplectomorphism $\tau_\kappa$ to $\mathcal{R}_{g,n-2}$ for $\kappa=1,2,3,4$. Letting $\iota_\kappa:D_\kappa\to\mathcal{R}_{g,n}$ denote the inclusions, these symplectomorphisms also satisfy $\tau_\kappa^*[\point]=\iota_\kappa^*[\point]$ and $\tau_\kappa^*[a_j]=\iota_\kappa^*[a_j]$ for $1\leq j\leq 2g$.
\end{lemma}
\begin{proof}
By Lemma \ref{lem_point_class}, we have $[\point]=-c_1(V_k)^2$ for all $k$. By Proposition \ref{prop_values_r_s}, we see that
$$\PD([\point])=([D^+_{12}]+[D^-_{12}])\cap([D^+_{23}]+[D^-_{23}])$$
for some choice of orientations for these submanifolds. Each of the four terms $D^\pm_{12}\cap D^\pm_{23}$ in the expansion is symplectomorphic to $\mathcal{R}_{g,n-2}$ by \ref{cor_symp_form_recursive}, through a map under which the universal bundle $\mathbf{E}^\ad$ pulls back to the restriction. The classes $[\point]$ and $[a_k]$ therefore also pull back, being defined through the universal bundle.
\end{proof}

\begin{cor} \label{cor_smaller_relations_pt}
Suppose $f$ is a polynomial in the $[a_j]$'s, $c_1(V_k)$'s, $[\omega]$, and $[\point]$ which is a relation in $H^*(\mathcal{R}_{g,n};\QQ)$. Then the polynomial $f\times[\point]^s$ is a relation in $H^*(\mathcal{R}_{g,n+2s};\QQ)$
\end{cor}

\subsection{The Classes $[\Sigma]$ and $[\omega]$}

We now make a brief digression on the relationship between the class $[\Sigma]$ and $[\omega]$. It will be convenient later to nail down precisely the linear combination (\ref{lin_comb_symp_class}). What we need is to describe the action of the flips $M_J$ on the class $[\Sigma]$. The issue is that $[\Sigma]$ is not preserved; by Proposition \ref{prop_symmetries_univ_bundle}, the pullback of the universal bundle $\mathbf{E}^\ad$ by $M_J$ is only an isomorphism away from the $x_k$'s. Fix an even $J$, and suppose $k\in J$. Recall that we have small disk neighborhoods $U_k$ with polar coordinate $(s_k,\theta_k)$ around each puncture $x_k$ and a trivialization of the $\U(2)$ bundle $E$. Let $\Sigma^\circ$ denote $\Sigma\setminus\cup_{l=1}^n U_l$. The map $M_J$ lifts to the space of flat connections via an $\SO(3)$ gauge transformation $u_J$, and the map (\ref{flip_on_univ_map}) is an isomorphism of $\mathbf{E}^\ad$ to  $(M_J\times\Id)^*\mathbf{E}^\ad$ over $\mathcal{R}_{g,n}\times\Sigma^*$. The bundles are also isomorphic when restricted to $\mathcal{R}_{g,n}\times U_k$. To see this, we note that the bundles are isomorphic when restricted to $\mathcal{R}_{g,n}\times\{x\}$ for any point $x\in\Sigma^\circ$, and so this is also true for $x\in U_k$. Since $\mathcal{R}_{g,n}\times U_k$ contracts to $\mathcal{R}_{g,n}$, the isomorphism is automatic. In fact, the proof of Lemma \ref{lem_point_class} shows that the restriction of $\mathbf{E}^\ad$ to $\mathcal{R}_{g,n}\times\{x_k\}$ is isomorphic to $\ubar\RR\oplus\mathbf{V}_k$, and so the restriction of  $(M_J\times\Id)^*\mathbf{E}^\ad$ to $\mathcal{R}_{g,n}\times\{x_k\}$ is isomorphic to $\ubar\RR\oplus\mathbf{V}_k^*$, which is isomorphic to $\ubar\RR\oplus\mathbf{V}_k$ as a real bundle. Hence, we can describe the new bundle  $(M_J\times\Id)^*\mathbf{E}^\ad$ as being obtained by cutting $\mathbf{E}^\ad$ along $\mathcal{R}_{g,n}\times\partial U_k$ for each $k$ and regluing with a ``clutching function''. We can compute this function as follows: it is the composition of the following circle of maps:

\xymatrixcolsep{3pc}
\begin{center} \mbox{
\xymatrix{
\mathbf{E}^\ad \ar[d]		     & 			                        & (M_J\times\Id)^*\mathbf{E}^\ad \ar[ll] \\
\ubar\RR\oplus\mathbf{V}_k  \ar[r]^{-1\oplus\Id}& \ubar\RR\oplus\mathbf{V}_k^* \ar[r] & \ubar\RR\oplus(M_J\times\Id)^*\mathbf{V}_k \ar[u] } }
\end{center}
The maps, beginning with the vector $\obar{(t,A,\theta)}$ in $\mathbf{E}^\ad|_x$ for $x=(1,\theta_k)\in\partial U_k$, $t=\ad\left(\begin{smallmatrix}s&-\obar z\\z&-s\end{smallmatrix}\right)\in\Ad E|_x$ and $A$ a flat connection, compose to:
\begin{equation} \begin{split}
\obar{\left(\ad\left(\begin{smallmatrix}s&-\obar z\\z&-s\end{smallmatrix}\right),A\right)}_{\theta_k}\mapsto&\;
\left(s,\obar{\left(\ad\left(\begin{smallmatrix}0&-\obar z\\z&0\end{smallmatrix}\right),A\right)}\right)_{\theta_k} \mapsto \\
&\;\left(-s,\obar{\left(\ad\left(\begin{smallmatrix}0&-z\\ \obar z&0\end{smallmatrix}\right),u_J\cdot A\right)}\right)_{\theta_k}\mapsto\hspace{0.5in}\text{(by equation (\ref{V_k_M_J_map}))} \\
&\;\obar{\left(\ad\left(\begin{smallmatrix}-s&-z\\ \obar z&s\end{smallmatrix}\right),u_J\cdot A\right)}_{\theta_k}\mapsto \\
&\;\obar{\left(\ad\left(\begin{smallmatrix}0&-e^{i\theta_k/2}\\e^{-i\theta_k/2}&0\end{smallmatrix}\right)\left(\begin{smallmatrix}-s&-z\\ \obar z&s\end{smallmatrix}\right)\left(\begin{smallmatrix}0&e^{i\theta_k/2}\\-e^{-i\theta_k/2}&0\end{smallmatrix}\right),A\right)}_{\theta_k} \\
=&\;\obar{\left(\ad\left(\begin{smallmatrix}s&-e^{i\theta_k}\obar z\\e^{-i\theta_k}z&-s\end{smallmatrix}\right),A\right)}_{\theta_k}
\end{split}
\end{equation}
We see that near $x_k$, $\mathbf{E}^\ad$ is isomorphic to the pullback of $\RR\oplus\mathbf{V}_k$ to $\mathcal{R}_{g,n}\times \wt U_k$ for a slightly larger disk neighborhood $\wt U_k\supset U_k$, and $(M_J\times\Id)^*\mathbf{E}^\ad$ is obtained by cutting this bundle along $\mathcal{R}_{g,n}\times\partial U_k$ and regluing $\ubar\RR\oplus\mathbf{V}_k$ to itself via $(s,v)\mapsto(s,e^{-i\theta_k}v)$, for each $k\in J$. It is not difficult to see that the characteristic classes of these two bundles are therefore related by:
\begin{lemma} \label{lem_sigma_vs_symp}
Upon slant product with $[\Sigma]$, the first Pontryagin classes of $\mathbf{E}^\ad$ and $(M_J\times\Id)^*\mathbf{E}^\ad$ are related by:
$$p_1\left((M_J\times\Id)^*\mathbf{E}^\ad\right)/[\Sigma]=p_1\left(\mathbf{E}^\ad\right)/[\Sigma]-2\sum_{k\in J}c_1(\mathbf{V}_k)$$
\end{lemma}
\begin{proof}
This is standard bundle theory and unwinding the definition of the slant product.
\end{proof}
We can use the lemma to determine the relationship between the cohomology classes $[\Sigma]$ and $[\omega]$. As before, by symmetry we know $[\omega]=s[\Sigma]+t\sum_{k=1}^nc_1(V_k)$ for some constants $s,t$. The class $\omega$ is invariant under flips, and the lemma implies that the only linear combinations of $[\Sigma]$ and $c_1(V_k)$ which are flip-invariant are scalar multiples of $[\Sigma]-\sum_{k=1}^nc_1(V_k)$. We conclude that
\begin{equation} \label{flip_invt_class}
[\omega]=A\left([\Sigma]+\sum_{k=1}^nc_1(V_k)\right)
\end{equation}
for some nonzero constant $A$. In fact, from \cite{weitsman_toric_I} we see that in our notation $A=-1/4$.

\subsection{Symplectic Volumes} \label{subsec_symp_vol}

Define the graded commutative polynomial algebra
$$\mathbb{A}_{g,n}:=\CC[\alpha,\beta,\delta_1,\ldots,\delta_n]\otimes\wedge^*[\psi_1,\ldots,\psi_{2g}],$$
where we assign $\alpha$ and $\delta_k$ degree 2, $\psi_j$ degree 3, and $\beta$ degree 4 (i.e. $\alpha$, $\beta$, and the $d_k$'s are commutative and the $\psi_j$'s anti-commute with each other). We denote by $\mathbb{H}_{g,n}$ the $\CC$-algebra $H^*(\mathcal{R}_{g,n};\CC)$. We can define a map $\Psi:\mathbb{A}_{g,n}\to\mathbb{H}_{g,n}$ via:
\begin{align*}
\alpha\mapsto&\;2[\omega] \\
\beta\mapsto&\;-\tfrac{1}{4}[\point] \\
\psi_j\mapsto&\;-\tfrac{1}{4}[a_j] \\
\delta_k\mapsto&\;\tfrac{1}{2}c_1(V_k)
\end{align*}
The fractional factors in front of each generator are used to make the presentation of the cohomology ring simpler and stem from the fact that we use the Pontryagin class of an adjoint universal bundle, rather than the second Chern class of a standard universal bundle. We remark that this notation gives
\begin{equation} \label{alpha_delta}
\Psi(\alpha)=-\tfrac{1}{2}[\Sigma]+\sum_{k=1}^n\Psi(\delta_k).
\end{equation}

What we have proved so far is that $\mathbb{H}_{g,n}$ is isomorphic to the ring $\mathbb{A}_{g,n}/\mathcal{I}_{g,n}$ for some ideal of relations $\mathcal{I}_{g,n}=\psi^{-1}(0)$, which includes the relations $\beta-\delta_k^2$ for each $k$. For each $g,n$, $n\geq 1$, there is a natural inclusion
$$\iota^{0,1}_{g,n}:\mathcal{R}_{g,n}\into\mathcal{R}_{g,n+2}$$
arising from the isomorphism $D^-_{n,n+1}\cap D^+_{n+1,n+2}\mapsto\mathcal{R}_{g,n}$, which by virtue of the results of this section has the property that pulling back the images under $\Psi$ of the generators $\alpha$, $\beta$, $\psi_j$, and $\delta_k$ (for $k=1,\ldots,n$) gives the corresponding generators for the smaller cohomology ring. It can also be checked that $\delta_{n+1}$ pulls back to $-\delta_n$ and $\delta_{n+2}$ pulls back to $\delta_n$. Let $\pi^{0,1}_{g,n}:\mathbb{A}_{g,n+2}\to\mathbb{A}_{g,n}$ denote the corresponding ring map. Then Corollary \ref{cor_smaller_relations} implies that $\pi(\mathcal{I}_{g,n})\subset\mathcal{I}_{g,n}$. More generally, there are inclusions
$$\iota_{g,n}^{r,s}:\mathcal{R}_{g,n}\into\mathcal{R}_{g+r,n+2s}$$
with corresponding ring maps $\pi^{r,s}_{g,n}:\mathbb{A}_{g+r,n+2s}\to\mathbb{A}_{g,n}$, under which
\begin{align*}
\delta_k\mapsto \begin{cases}(-1)^{k-n}\delta_n\text{, }&k>n \\ \delta_k\text{, }&k\leq n\end{cases},\hspace{36pt}
\alpha_j\mapsto \begin{cases}\pm 0\text{, }&j>2g \\ \alpha_j\text{, }&j\leq 2g\end{cases}
\end{align*}
and we have the inclusion $\pi^{r,s}_{g,n}(\mathcal{I}_{g+r,n+2s})\subset\mathcal{I}_{g,n}$. This is a direct consequence of Corollary \ref{cor_smaller_relations} and \cite{nitsure_cohom}. This encapsulates the inductive structure of the moduli spaces, and in the case of no marked points is well known.

\begin{nonumnotation}
From now on, we denote by a hat the image under $\Psi$ in $\mathbb{H}_{g,n}$ of a generator in $\mathbb{A}_{g,n}$ by a hat. For example, we have $\wh\beta=\wh\delta_k^2$ for all $k$.
\end{nonumnotation}

There is an immediate relation in the cohomology ring resulting from comparing different versions of the volume class on $\mathcal{R}_{g,n}$. Namely, if we let $2D=\dim_\RR(\mathcal{R}_{g,n})=6g+2n-6$ and set $e=D\mod{2}=0,1$, then the images of $\wh\alpha^D$ and $\wh\beta^{\floor{D/2}}\wh\alpha^e$ are both multiples of eachother. We are thus interested in pairings of the form $\pair{\wh\alpha^r\wh\beta^s}{\mathcal{R}_{g,n}}$ for $r+2s=D$. Lemma \ref{lem_point_class_pd} allows us to compute these, once we know the symplectic volume formula (\ref{symp_volume}), at least for $s\leq m$:
\begin{equation} \label{ab_pairings}
\langle\wh\alpha^r\wh\beta^s,\mathcal{R}_{g,n}\rangle=\frac{(3g+n-2s-3)!}{(2g+n-2s-3)!}\abs{E_{2g+n-2s-3}}=\frac{r!}{(r-g)!}\abs{E_{r-g}}.
\end{equation}
However, there will be relations of smaller degree involving only $\wh\alpha$ and $\wh\beta$, which arises from this formula and Poincar\'e duality. The role of Poincar\'e duality is in the following statement: if $f\in\mathbb{A}_{g,n}$ with $\deg f=r$ and $\pair{\Psi(f)\Psi(f'),\mathcal{R}_{g,n}}=0$ for all $f'$ of degree $2D-f'$, then $f$ is a relation in $\mathcal{I}_{g,n}$. It also implies that there must be at least one relation of degree $D+2$, or just over half the dimension, since the Betti numbers must be symmetric about the middle dimension and the dimension of the degree $d$ part of $\mathbb{A}_{g,n}$ strictly increases with $d$.

\section{The Case $g=0$}

Recall that we have set $n=2m+1>1$, and we now set $g=0$. In this case $\mathbb{A}_{0,n}$ is generated by $\alpha$, $\beta$, and the $\delta_k$'s, and $D=n-3=2m-2$ (in fact, $\beta$ is redundant).

\begin{prop} \label{prop_genus_0_first_rel}
There is a unique polynomial $r_{0,n}(\alpha,\beta)$ in $\mathcal{I}_{0,n}$ of degree $2m$ monic with respect to $\alpha$. It is obtained via the recursion
\begin{equation} \label{genus_0_recurrence}
\begin{split}
r_{0,2m+3}(\alpha,\beta)=&\;\alpha\cdot r_{0,2m+1}(\alpha,\beta)-m^2\beta\cdot r_{0,2m-1}(\alpha,\beta) \\
r_{0,1}(\alpha,\beta)=&\;1,\hspace{18pt}r_{0,3}(\alpha,\beta)=\alpha
\end{split}
\end{equation}
\end{prop}
\begin{proof}
We will argue by induction. We know a relation of degree $2m$ must exist due to Poincar\'e duality. The content of the lemma is that there is a relation involving only $\alpha$ and $\beta$, that it is unique, and that there is a recursive formula in $n$ which it satisfies. To see that there is a relation in $\alpha$ and $\beta$ and that it is unique, let $W_m$ denote the vector space of possible polynomials $r_{0,n}(\alpha,\beta)$ of degree $2m$. Letting $e$ denote the remainder upon dividing $m$ by 2, such a polynomial looks like
$$r_{0,n}(\alpha,\beta)=A_m\alpha^m+A_{m-2}\alpha^{m-2}\beta+\ldots+A_e\alpha^e\beta^{\floor{m/2}}$$
so that $\dim W_m=\floor{m/2}+1$. For each of the $\floor{m/2}$ choices of $s=\floor{m/2}-1,\floor{m/2}-2,\ldots,0$, there is a monomial $\alpha^{m-2-2s}\beta^s$ of complementary degree $2m-4$. We get $\floor{m/2}$ linear functionals
$$\ell_s:r_{0,n}\mapsto\pair{r_{0,n}\wh\alpha^{m-2-2s}\wh\beta^s}{\mathcal{R}_{0,n}}$$
and the relation we seek will lie in the kernel of each of them. Since $\dim W_m$ has one greater dimension than this collection of functionals, a (nonzero) $r_{0,n}$ will certainly exist.

We can rephrase this as saying the vector $(A_e,\ldots,A_{m-2},A_m)\in\RR^{\floor{m/2}+1}$ of coefficients is in the kernel of the $(\floor{m/2}-1)\times\floor{m/2}$ matrix $(E_{ij})$ where $$E_{ij}=\pair{\wh\alpha^{2e+2j+2i}\wh\beta^{\floor{m/2}-1-i}}{\mathcal{R}_{0,n}}$$
with $i=0,\ldots,\floor{m/2}-1$ and $j=0,\ldots,\floor{m/2}$. By the formula (\ref{ab_pairings}), we have in the case $g=0$ the very simple expression $E_{ij}=\abs{E_{2e+2i+2j}}$. This is an example of a ``Hankel'' matrix, a type of matrix which arises when studying the so-called ``moment problem'' in connection with the theory of orthogonal polynomials and continued fractions (here, with the sequence of moments $\abs{E_0},\abs{E_2},\abs{E_4},\ldots$). This theory, along with the known continued fraction expansion for the formal generating function $\sum_{i=0}^\oo \abs{E_i}x^i$, implies that the polynomial $r_{0,n}$ is unique up to scale and satisfies the beautiful recurrence relation (\ref{genus_0_recurrence}). We relegate the proof of this formula to the appendix.

The polynomial $r_{0,n}$ is rigged to pair to 0 with each of the complementary degree monomials involving just $\wh\alpha$ and $\wh\beta$. It remains to check that this $r_{0,n}$ pairs to 0 with complementary polynomials in not just $\wh\alpha$ and $\wh\beta$ but also the $\delta_k$'s. Because of the relation $\wh\delta_k^2=4\wh\beta$, we simply need to check that $r_{0,n}$ pairs to 0 with terms of the form $\wh\alpha^r\wh\beta^s\wh\delta_{k_1}\cdots\wh\delta_{k_t}$ with $r+2s+t=m-2$ and the $k_i$'s distinct. By inductive hypothesis, we have the relation $r_{0,n-2k}$ in the ideal $\mathcal{I}_{0,n-2s}$. The recurrence relation implies that $r_{0,n}$ is also a relation in $\mathcal{I}_{0,n-2s}$. By Corollary \ref{cor_smaller_relations} the product $\delta_{k_1}\cdots\delta_{k_t}\cdot r_{0,n}$ is a relation in $\mathcal{I}_{0,n}$, which is enough to ensure the vanishing of all pairings in $\mathbb{A}_{0,n}$ with $r_{0,n}$.
\end{proof}

From this, we can write down a large collection of relations which must hold in the ring $\mathbb{H}_{0,n}.$

\begin{cor} \label{cor_genus_0_relations}
For each $J\subset\{1,\ldots,n\}$ with $\abs{J}=s\leq m$, the polynomial
\begin{equation} \label{genus_0_rel}
R^J_{0,n}=r_{0,n-2s}(\alpha,\beta)\cdot\prod_{k\in J}\delta_k
\end{equation}
is in the ideal of relations $\mathcal{I}_{0,n}$
\end{cor}
\begin{proof}
Simply combine Proposition \ref{prop_genus_0_first_rel} and Corollary \ref{cor_smaller_relations}.
\end{proof}

All that remains to show is that this is a complete set of relations. For this we mimic the approach in \cite{siebert_tian} and describe an explicit basis for $H^*(\mathcal{R}_{0,n})$, and show that any other monomial can be expressed a linear combination of monomials in the basis and the relations $R^J_{0,n}$. In what follows, we denote $\ubar\delta^J=\prod_{k\in J}\delta_k$.

\begin{lemma} \label{lem_genus_0_basis}
Let $\mathcal{S}_{0,n}$ denote the collection of monomials $\alpha^a\beta^b\ubar\delta^J$ with $a+b+\abs{J}<m$. Then any other monomial $\alpha^{a'}\beta^{b'}\ubar\delta^{J'}$ with $a'+b'+\abs{J'}\geq m$ can be reduced to a linear combination of monomials in $\mathcal{S}_{0,n}$ and the relations $R^J_{0,n}$ and $\delta_i^2-\beta$.
\end{lemma}
\begin{proof}
We can certainly assume $\abs{J}<m$, since $\ubar\delta^J$ is a relation if $\abs{J}\geq m$. We first treat the case $J=\emptyset$. Because the leading term (with respect to $\alpha$) of $R^\emptyset_{0,n}=r_{0,n}$ is $\alpha^m$ and all other terms are monomials in $\alpha$ and $\beta$ with lower exponent sum, we can certainly reduce the monomial $\alpha^a\beta^b$ for $a+b\geq m$ to a linear combination of $\alpha^r\beta^s$ with $r+s<m$. Hence, we suppose $\abs{J}\geq 1$. Now, suppose that $\phi\in\Mod_{0,n}$. If we can reduce $M_\phi^*(z)$ for a monomial $z$, then we can certainly reduce $z$, because the collection of relations $R^J_{0,n}$ is preserved by the mapping class group action. Since the mapping class group action serves to permute the $\delta_i$'s, without a loss of generality we may prove the lemma for monomials with $J=J'\cup\{\delta_{n}\}$ where $J'\subset\{1,\ldots,n-2\}$.

We argue by induction on $n$. Suppose the lemma is true for $n$ and for $J\subset\{1,\ldots,n+2\}$ assume that $J=J'\cup\{\delta_{n+2}\}$ with $J'\subset\{1,\ldots,n\}$. Suppose that $a+b+\abs{J}\geq m+1$. Then $a+b+\abs{J'}\geq m$ and so by inductive hypothesis, the monomial $\alpha^a\beta^b\ubar\delta^{J'}$ may be reduced to a linear combination of relations $R^K_{0,n}$, $\delta_i^2-\beta$, and monomials in $\mathcal{S}_{0,n}$. Multiplying $R^K_{0,n}$ by $\delta_{n+2}$ gives the relation $R^{K\cup\{\delta_{n+2}\}}_{0,n+2}$, and so multiplying this linear combination by $\delta_{n+2}$ gives a reduction for our monomial $\alpha^a\beta^b\ubar\delta^J$, as desired. This completes the proof.
\end{proof}

\begin{prop} \label{prop_genus_0_complete}
Along with the relations $\delta_i^2-\beta$, the set of relations $R_{0,n}^J$ is a complete set.
\end{prop}
\begin{proof}
We saw in \S\ref{sec_betti_numbers} that the Poincar\'e polynomial of $\mathcal{R}_{0,n}$ agreed with that of the graded algebra $\CC[\alpha,\beta,\delta_1,\ldots,\delta_n]/(\delta_i^2)$ up to the middle dimension. The algebra
$$\CC[\alpha,\beta,\delta_1,\ldots,\delta_n]/(\delta_i^2-\beta)$$
has the same Poincar\'e polynomial. Hence, there can be no relations other than $\delta_i^2-\beta$ in $H^*(\mathcal{R}_{0,n})$ below degree $2m$, which is the degree of $R_{0,n}^J$. Now, by Lemma \ref{lem_genus_0_basis} the relations imply that $\mathcal{S}_{0,n}$ as above contains a basis. But it is a simple matter to check that the if $\mathcal{S}_{0,n}(d)$ denotes the number of monomials of degree $d$, then $\mathcal{S}_{0,n}(d)$ is the same as the the dimension of the degree $d$ part of $\CC[\alpha,\beta,\delta_1,\ldots,\delta_n]/(\delta_i^2-\beta)$ up to the middle dimension. Moreover, it is easy to check that $\mathcal{S}_{0,n}(2n-6-d)= \mathcal{S}_{0,n}(d)$, and so by Poincar\'e duality, we must have that $\mathcal{S}_{0,n}$ actually is a basis, and so we have a complete set of relations.
\end{proof}

We have therefore proved Theorem \ref{thm_cohom_ring_genus_0}.
\appendix
\section[Appendix]{Euler Numbers, Orthogonal Polynomials, and Continued Fractions} \label{sec_appendix}

We owe the reader a discussion of how to arrive at the recursive relation (\ref{genus_0_recurrence}) for the relations in the cohomology ring, given that the top pairings of the generators $\alpha$ and $\beta$ are the Euler numbers $E_n$. This requires a brief digression on orthogonal polynomials, and an analysis of the ordinary generating function for the numbers $E_n$.

\vspace{12pt}

\noindent\textbf{Orthogonal Polynomials.} We begin by supposing we have a measure $\mu$ on the interval $[a,b]\subset\RR$, which we suppose for simplicity is given by integrating against a continuous, nonnegative weighting function $w(x)$:
$$\int_a^b f(x) d\mu=\int_a^b f(x)w(x)dx.$$
This measure provides a linear functional $\mathcal{L}_\mu$, as well as an inner product and norm on the set of $\mu$-integrable functions on $[a,b]$:
\begin{align*}
\mathcal{L}_\mu(f):=&\;\int_a^b f(x) d\mu \\
\pair{f}{g}_\mu:=&\;\int_a^bf(x)g(x)d\mu \\
\norm{f}_\mu:=&\;\left(\pair{f}{f}_\mu\right)^{\tfrac{1}{2}}
\end{align*}

\begin{definition} \label{def_moments}
The sequence of numbers
$$c_n=\mathcal{L}(x^n)=\int_a^b x^nd\mu$$
is called the sequence of \emph{moments} for the measure $\mu$.
\end{definition}

Given $\mu$, a useful collection of data is its sequence of orthogonal polynomials, which provide a convenient basis for the set of integrable functions on $[a,b]$.

\begin{definition}
A sequence of \emph{monic orthogonal polynomials} for the measure $\mu$ is a sequence of polynomials $\{p_n(x)\}$ satisfying
\begin{enumerate}[(i)]
\item $p_n$ is monic and $\deg(p_n)=n$
\item $\pair{p_n}{p_m}_\mu=0$ for $n\neq m$ and $\norm{p_n}_\mu>0$ for all $n$.
\end{enumerate}
\end{definition}

It is very easy to prove:

\begin{lemma} \label{lem_ortho_poly}
Given a weighting function $w(x)$ and associated measure $\mu$, a sequence of monic orthogonal polynomials $\{p_n(x)\}$ exists and is unique.
\end{lemma}

We will thus refer to \emph{the} sequence of monic orthogonal polynomials for a given measure $\mu$. The theory of orthogonal polynomials is old and well-understood. One has the following famous result, known as the three-term recurrence:

\begin{thm} \label{thm_three_term}
(\cite{barry_ortho}, Theorem 5) For the measure $\mu$, the monic orthogonal polynomials $p_n(x)$ satisfy the following recurrence relation:
\begin{equation} \label{three_term}
p_{n+1}(x)=(x-\alpha_n)p_n(x)-\beta_np_{n-1}(x)
\end{equation}
where we initialize $p_0(x)=1$, $p_{-1}(x)=0$, and where $\alpha_n$ and $\beta_n$ are constants depending on $\mu$.
\end{thm}

In fact, the numbers $\alpha_n$ and $\beta_n$ can be computed as follows:
\begin{equation*}
\alpha_n=\frac{\pair{xp_n(x)}{p_n(x)}_\mu}{\norm{p_n(x)}_\mu}, \hspace{12pt} \beta_n=\frac{\norm{p_n(x)}_\mu}{\norm{p_{n-1}(x)}_\mu}
\end{equation*}

There is a fascinating relationship between the collection of moments $\{c_n\}$ and the coefficients $\alpha_n$ and $\beta_n$.

\begin{thm} \label{thm_cont_frac}
Let $F_\mu(x)$ be the ordinary generating function for the moments $c_n=\mathcal{L}_\mu(x^n)$:
$$F_\mu(x)=\sum_{n=0}^\oo c_nx^n$$
Then the coefficients $\alpha_n$ and $\beta_n$ from (\ref{three_term}) and $F_\mu(x)$ satisfy the continued fraction identity:
\begin{equation} \label{cont_frac}
F_\mu(x)=	\cfrac{c_0}{1-\alpha_0x-
			\cfrac{\beta_1x^2}{1-\alpha_1x-
			\cfrac{\beta_2x^2}{1-\alpha_2x-
			\cfrac{\beta_3x^2}{1-\alpha_3x-\cdots}}}}
\end{equation}
\end{thm}

\vspace{12pt}

\noindent\textbf{The Euler Numbers.} The numbers $E_n$ are defined via:
\begin{equation} \label{euler_numbers}
\sech(z)=\frac{1}{\cosh(z)}=\sum_{n=0}^\oo\frac{E_{n}}{n!}z^{n}.
\end{equation}
Since $\sech(z)$ is an even function, we have $E_n=0$ for $n$ odd. The sequence has first few terms (beginning with the 0th term):
$$\{E_n\}=1,0,-1,0,5,0,-61,0,1385,\ldots$$
where the nonzero entries are alternatively positive and negative. It is shown in \cite{flajolet_cont_frac} that the ordinary generating function for the sequence $\{\abs{E_{2n}}\}$ of absolute values of nonzero Euler numbers has the following remarkable continued fraction expansion:
\begin{equation} \label{euler_gen_func}
E(z)=\sum_{n=0}^\oo \abs{E_{2n}}z^{2n}= \cfrac{1}{1-
										\cfrac{1^2z^2}{1-
										\cfrac{2^2z^2}{1-
										\cfrac{3^2z^2}{1-\ldots}}}}
\end{equation}
Here we must use the absolute value sign on the Euler numbers to agree with the conventions of \cite{flajolet_cont_frac}.

We now tackle our main problem, which is understanding the relations in the cohomology ring of the moduli space $\mathcal{R}_{0,2m+1}$. We have already shown that there is a relation $r_{0,2m+1}(\alpha,\beta)$ in the generators $\alpha$ and $\beta$ which is of degree $m$ in $\alpha$ and $\floor{m/2}$ in $\beta$. Set $e=m-2\floor{m/2}$, the remainder 0,1 of $m$ when divided by 2. As in the proof of Proposition \ref{prop_genus_0_first_rel}, if we write
$$r_{0,2m+1}(\alpha,\beta)=A_m\alpha^m+A_{m-2}\alpha^{m-2}\beta+\ldots+A_e\alpha^e\beta^{\floor{m/2}}$$
then the $(\floor{m/2}+1)$-vector $(A_m,A_{m-2},\ldots,A_e)$ is in the kernel of the matrix $E_{ij}$ for $i=0,\ldots,\floor{m/2}-1$ and $j=0,\ldots,\floor{m/2}$ with
\begin{equation} \label{euler_mat}
E_{ij}=\abs{E_{2e+2i+2j}}.
\end{equation}
Let us take this as a definition of $r_{0,2m+1}(\alpha,\beta)$, under the additional constraint that we take $r_{0,2m+1}(\alpha,\beta)$ to be monic in $\alpha$.

\begin{thm} \label{thm_recur_rel}
The polynomials $r_{0,2m+1}(\alpha,\beta)$ satisfy the recurrence (\ref{genus_0_recurrence}).
\end{thm}
\begin{proof}
Suppose that $\mu$ is a measure on an interval $[a,b]$ whose moments are given by
$$\mathcal{L}_\mu(x^n)=\abs{E_n}.$$
Define the de-homogenized single-variable polynomial:
\begin{align*}
s_{m}(x):=&\;r_{0,2m+1}(x,1).
\end{align*}
Evaluating the matrix $E_{ij}$ on the $(\floor{m/2}+1)$-vector $(A_m,A_{m-2},\ldots,A_e)$ gives the $\floor{m/2}$-vector
$$\left(\pair{x^e}{s_{m}(x)}_\mu,\pair{x^{e+2}}{s_{m}(x)}_\mu,\ldots,\pair{x^{m-2}}{s_{m}(x)}_\mu\right),$$
which we assume vanishes. Hence, the pairing of $s_m(x)$ (which is degree $m$ in $x$) with any polynomial in $x$ of lower degree having only terms with the same degree parity is zero. Of course, pairing with monomials with the opposite parity gives 0 as well. This implies $s_m(x)$ is orthogonal to $s_k(x)$ for $k<m$, and so by induction the polynomial sequence $\{s_m(x)\}$ is monic orthogonal for $\mu$. By Theorem \ref{thm_cont_frac}, we see that $s_m(x)$ satisfies the recurrence
$$s_{m+1}(x)=(x-\alpha_m)s_m(x)-\beta_ms_{m-1}(x)$$
where $\alpha_m$ and $\beta_m$ are the coeffiecients in the continued fraction \ref{cont_frac}. But by definition, the generating function to use is given by $E(z)$ as in (\ref{euler_gen_func}), and so we see that $\alpha_m=0$ and $\beta_m=-m^2$. Hence:
$$s_{m+1}(x)=xs_m(x)-m^2s_{m-1}(x),$$
and since $r_{0,2m+1}(\alpha,\beta)$ is homogeneous in $\alpha$ and $\beta$ (with these variables assigned degrees 2 and 4, respectively), we obtain the recursion (\ref{genus_0_recurrence}).
\end{proof}

\bibliography{thesis_bib}{}
\bibliographystyle{plain}

\end{document}